\documentclass[11pt,twoside,reqno]{amsart}
\usepackage{amsfonts}
\usepackage{amsmath}
\usepackage{amsmath,bm}
\usepackage{amssymb}
\usepackage{amsthm}
\usepackage{amsmath, amsfonts, amsthm, amssymb,amscd, graphicx, amscd}
\usepackage{float,epsf}

\usepackage{enumerate}
\usepackage{tikz}
\usepackage{color}
\usepackage{srcltx}
\usepackage{graphicx} 
\usepackage{epstopdf}
\usepackage{psfrag}

\usepackage{geometry}
\usepackage{mathrsfs}

\newcommand{\gam}{\gamma}
\newcommand{\ol}{\overline}

\newcommand{\der}{\partial}

\newcommand {\rx}{{\rm x}}

\theoremstyle{problem}
\newtheorem{problem}{Problem}
\newtheorem{theorem}{Theorem}
\newtheorem{definition}{Definition}

\newtheorem{lemma}[theorem]{Lemma}

\newtheorem{prop}[theorem]{Proposition}

\theoremstyle{remark}
\newtheorem{remark}{Remark}

\numberwithin{equation}{section}

\allowdisplaybreaks

\begin{document}
		\title{Global Uniqueness of Subsonic Flows for the Steady Euler-Poisson System }

\author{Myoungjean Bae}
\address{Department of Mathematical Sciences, KAIST, 291 Daehak-ro, Yuseong-gu, Daejeon, 43141, Korea}
\email{mjbae@kaist.ac.kr}

\author{Ben Duan}
\address{School of Mathematics, Jilin University, Changchun, Jilin Province, 130012, China}
\email{bduan@jlu.edu.cn}

\author{Chunjing Xie}
\address{School of Mathematical Sciences,  Ministry of Education Key Laboratory of Scientific and Engineering Computing, CMA-Shanghai, Shanghai Jiao Tong University, Shanghai, 200240, China}
\email{cjxie@sjtu.edu.cn}

		\date{\today}


		\maketitle

\newcommand{\shock}{\Gamma_{\rm{shock}}}
\newcommand{\Gam}{\Gamma}
\newcommand{\Om}{\Omega}
\newcommand{\eps}{\varepsilon}
\newcommand{\R}{\mathbb{R}}
\newcommand{\alp}{\alpha}
\newcommand{\om}{\omega}
\newcommand{\mfrak}{\mathfrak}
\newcommand{\vphi}{\varphi}
\begin{abstract}
We prove the global uniqueness of multidimensional subsonic flows for the steady Euler--Poisson system in a bounded nozzle in the sense that uniqueness holds without restricting solutions to be small perturbations of a background state. The proof is based on a convexity property of the set of subsonic states and energy estimates.

\end{abstract}


\section{Introduction and main results}

The Euler--Poisson system is a macroscopic continuum mechanics model that arises naturally in various areas of mathematical physics, including plasma physics, semiconductor theory, and astrophysics. The general form of the system can be written as 
\begin{equation}
    \begin{cases}
 \partial_t \rho+   {\rm div}(\rho {\bf u})=0\\
\partial_t(\rho {\bf u})+    {\rm div}(\rho {\bf u}\otimes {\bf u}+p\mathbb{I}_n)=\rho\nabla\Phi\quad(\mathbb{I}_n:\text{$n\times n$ identity matrix})\\
    \Delta \Phi=w \rho-b(\rx),
    \end{cases}
\end{equation}
where $\rho$ and ${\bf u}$ are the density and velocity of the fluid, $\Phi$ is the electric  or gravitational potential. For flows in  semiconductor models or the self-gravitation model, $w$ is a constant equal to $1$ or $-1$, respectively. $b(x)$ represents the background density of the ions for flows in semiconductor models, while it is equal to $0$ for the self-gravitating model. In this paper, we can even study a more general system with $w$ and $b$ as  functions of $x_1$. Mathematically, the system exhibits a strong nonlinear coupling between fluid variables and a nonlocal potential force, and makes the analysis for the associated unsteady flows very challenging.

In recent years, there were quite a few studies for steady solutions of the Euler-Poisson system, i.e., the solutions of the system
\begin{equation}
\label{E-P general}
    \begin{cases}
    {\rm div}(\rho {\bf u})=0,\\
    {\rm div}(\rho {\bf u}\otimes {\bf u}+p\mathbb{I}_n)=\rho\nabla\Phi,\\
    \Delta \Phi=w(\rx)\rho-b(\rx).
    \end{cases}
\end{equation}
It is well known that the multidimensional steady Euler system is a mixed type system. This makes the analysis of the steady Euler-Poisson system very difficult. The one dimensional solutions of the Euler-Poisson system were first studied in \cite{Markphase}. The vanishing viscosity limit of the one-dimensional steady solutions was established in \cite{Gamba}. In \cite{LuoXin}, one-dimensional steady solutions of the Euler-Poisson system were studied intensively. Later, the structural and dynamical stability of transonic shocks was studied in \cite{luo2011stability}. The multidimensional Euler-Poisson system is much more difficult to investigate. In a series of papers, the structural stability of subsonic solutions for the Euler-Poisson system was achieved in \cite{BDX_SIMA, bae2015two, BDX_ARMA16, bae20183}. More precisely, in the neighborhood of the one-dimensional subsonic solutions, it is proved that there is a unique multidimensional subsonic solution under suitable multidimensional perturbation.  The uniqueness is proved under the condition that the solutions are close to some one-dimensional solutions. Hence, the result is local and does not extend to general solutions. 

The steady solutions for the Euler-Poisson system for self-gravitating stars were studied in \cite{
DengLiuYangYao, DengYang, JangStraussWu,
LiBao_DCDS,
LuoSmoller2004, 
LuoSmoller2009,
Wu_ARMA, 
Yuan_SCM}. The structural stability of subsonic solutions for self-gravitating flows was proved in \cite{bae2015two}.

A natural problem is whether there is a unique subsonic solution globally under suitable boundary conditions. Usually, even the uniqueness for a single nonlinear partial differential equation is quite challenging. For subsonic flows with zero vorticity, compressible Euler equations can be written into a single equation for velocity potential. The nice feature for potential equation of subsonic Euler flows is that the maximum principle can be applied, which yields the global uniqueness for subsonic flows in nozzles under some natural boundary conditions, see \cite{XuYuan}.
A fundamental difficulty in studying global uniqueness of subsonic solutions of the Euler-Poisson equations lies in the nature of the equations. In the subsonic regime, the Euler--Poisson system reduces to a coupled nonlinear elliptic system. Unlike scalar elliptic equations, for which the maximum principle often provides a powerful tool to establish uniqueness, general elliptic systems do not admit a maximum principle. As a consequence, there is, in general, no available theory guaranteeing global uniqueness for elliptic systems. From this perspective, establishing global uniqueness for the Euler--Poisson system requires ideas beyond standard elliptic theory.
This is a major goal of this paper.

\subsection{Problems and main results}
Let $\Lambda$ be an open, bounded, and connected set in $\R^{n-1}$($n=2$ or $3$). For a constant $L>0$, define a nozzle
\[
\mathcal{N}_L:=\{{\rm x}=(x_1, {\rm x}')\in \R^n: x_1\in (0,L),\,\,{\rm x}'\in \Lambda\}.
\]
The nozzle boundary $\der\mathcal{N}_L$ consists of the entrance
\[\Gam_0:=\{{\rm x}=(x_1, {\rm x}')\in \der \mathcal{N}_L: x_1=0,\,\,{\rm x}'\in \ol{\Lambda}\},\]
the exit
\[\Gam_L:=\{{\rm x}=(x_1, {\rm x}')\in \der \mathcal{N}_L: x_1=L,\,\,{\rm x}'\in \ol{\Lambda}\},\]
and the wall
\[\Gam_w:=\{{\rm x}=(x_1, {\rm x}')\in \der \mathcal{N}_L: x_1\in (0,L),\,\,{\rm x}'\in \der{\Lambda}\}.\]
Consider the steady isentropic Euler-Poisson system
\eqref{E-P general}
for smooth functions $w:\ol{\mathcal{N}_L}\rightarrow \R$, and $b:\ol{\mathcal{N}_L}\rightarrow \R$.
In this paper, we consider two cases.

{\it Case 1. Euler-Poisson system of electric field type}:
\[\min_{\rx\in \ol{\mathcal{N}_L}}w(\rx)>0.\]
The particular case \[
w(\rx)\equiv 1,\quad b(\rx)=\bar{\rho}_I
\]
with a positive constant $\bar{\rho}_I>0$ is an important model to describe the motion of electrons in semiconductor devices, and was investigated in \cite{bae2015two, BDX_ARMA16, bae20183}. The main results of these works show the structural stability of one-dimensional subsonic solutions under small perturbations of the boundary conditions or the domain.

{\it Case 2. Euler-Poisson system of self gravitation type}:
\[\max_{\rx\in \ol{\mathcal{N}_L}}w(\rx)<0.\]
The particular case
\[
w(\rx)\equiv -1,\quad b(\rx)\equiv 0,
\]
can be used to describe the motion of flows with self-gravitation, and was studied extensively in the last two decades. 

We assume that $p\in C^0([0,\infty))\cap C^\infty((0,\infty))$ satisfies the following properties:
\[
p(0)=0,\quad \lim_{\rho\to \infty}p(\rho)=\infty;\quad 
p'(\rho)>0 \,\, \text{and}\,\, p''(\rho)\ge 0\,\,\text{for all}\,\, \rho>0.
\]
A typical example is $p(\rho)=\rho^{\gamma}$ with $\gamma\ge 1$, which corresponds to a polytropic gas.


The Mach number $M$ is defined by
\[M:=\frac{|{\bf u}|}{c(\rho)}\quad\text{for $c(\rho):=\sqrt{p'(\rho)}$}.\]
Here, $c(\rho)$ is called the local sound speed. If $M<1$, the corresponding state is called subsonic; if $M>1$, it is called supersonic; and if $M=1$, it is called sonic. In this paper, we study the existence and global uniqueness of subsonic solutions to the steady Euler-Poisson system. In particular, we focus on the case of zero vorticity ($\nabla\times {\bf u}={\bf 0}$), for which we adopt different formulations in (Case 1) and (Case 2), respectively. These formulations were first introduced in \cite{bae2015two} and \cite{BDX_ARMA16}.
If $\nabla\times{\bf u}={\bf 0}$, then we can rewrite the momentum equation (the second equation in \eqref{E-P general}) as
\[\rho \nabla\left(\frac 12 |{\bf u}|^2+i(\rho)-\Phi\right)=0,\]
where $i(\rho)$ is called enthalpy and is defined by
\begin{equation}
\label{definition-enthalpy}
i(\rho):=\int_1^{\rho}\frac{p'(\varrho)}{\varrho}\,d\varrho.
\end{equation}
Since we seek solutions satisfying $\rho>0$, it holds
\[
\frac12 |{\bf u}|^2+i(\rho)-\Phi=\text{constant}.
\]

For smooth functions $w:[0,L] \rightarrow \R$ and $b:[0,L] \rightarrow \R $, let us consider the irrotational Euler-Poisson system
\begin{equation}
\label{EP-general}
    \begin{cases}
    {\rm div}(\rho {\bf u})=0,\\
    \nabla\times {\bf u}={\bf 0},\\
    \frac12 |{\bf u}|^2+i(\rho)-\Phi=\text{constant},\\
    \Delta \Phi=w(x_1)\rho-b(x_1).
    \end{cases}
\end{equation}
Clearly, $(\rho,{\bf u},\Phi)=(\bar{\rho}(x_1), \bar u_1(x_1){\bf e}_1, \bar{\Phi}(x_1))$
solves \eqref{EP-general} in $\mathcal{N}_L$ if and only if, for $0<x_1<L$, $(\bar{\rho}, \bar{u}_1, \bar{\Phi})$ satisfies
\begin{equation}
    \label{EP-1d}
    \begin{cases}
    (\bar{\rho} \bar{u}_1)'=0,\\
(\bar{\rho}\bar{u}_1^2+p(\bar{\rho}))'=\bar{\rho}\bar{\Phi}',\\
\bar{\Phi}''=w(x_1)\bar{\rho}-b(x_1),   
    \end{cases}
\end{equation}
where $'$ denotes the derivative with respect to $x_1$. Suppose that \eqref{EP-1d} has a smooth solution on $[0,L]$ and it satisfies
\[\bar{\rho}\bar u_1=J\quad\text{on $[0,L]$}\]
for a constant $J>0$. Then \eqref{EP-1d} can be reduced to
\begin{equation}
    \label{EP-1d-reduced}
    \begin{cases}
\bar{\rho}'=\frac{\bar{\rho}\bar E}{p'(\bar{\rho})-\frac{J^2}{\bar{\rho}^2}},\\
    \bar E'=w(x_1)\bar{\rho}
    -b(x_1),
    \end{cases}\quad\text{on $(0,L)$,}
\end{equation}
where $\bar{E}:=\bar{\Phi}'$.
The unique existence of smooth solutions to the system \eqref{EP-1d-reduced} has been rigorously analyzed in
\cite{BDX_ARMA16, LuoXin} for the case
$w(x_1)\equiv 1, \quad b(x_1)\equiv b_0$
with a constant $b_0$, and in \cite{bae2015two} for the case
$
w(x_1)\equiv -1, \quad b(x_1)\equiv 0$.

In this paper, we assume that \eqref{EP-1d-reduced} has a smooth subsonic solution on a given interval $[0,L]$ and prove that the one-dimensional solution is structurally stable in the sense that if we prescribe boundary conditions on $\partial\mathcal{N}_L$ as small perturbations of the one-dimensional solution, then there exists a unique subsonic solution to \eqref{EP-general} in $\mathcal{N}_L$.
Moreover, we establish the global uniqueness of subsonic solutions in the admissible function class. More precisely, our uniqueness result is not restricted to a local perturbative regime around the one-dimensional solution: if two subsonic solutions to \eqref{EP-general} exist in $\mathcal{N}_L$ under the prescribed boundary conditions, then they must coincide, even if neither of them is assumed a priori to be a small perturbation of the one-dimensional solution.

\begin{problem}
\label{main problem}
Suppose that $(\bar{\rho}, {\bar u}_1{\bf e}_1, \bar{\Phi})$ is a smooth subsonic solution of \eqref{EP-1d} on $[0,L]$ for a constant $J>0$. Given functions
\[g_0:\ol{\Gam_0}\rightarrow \R,\quad h_0:\ol{\Gam_0}\rightarrow \R,\quad v_L:\ol{\Gam_L}\rightarrow \R,\]
find a subsonic solution $(\rho, {\bf u}, \Phi)$ of the system 
\begin{equation}
\label{EP-general-b}
    \begin{cases}
    {\rm div}(\rho {\bf u})=0\\
    \nabla\times {\bf u}={\bf 0}\\
    \frac12 |{\bf u}|^2+i(\rho)-\Phi=\frac 12 \bar{u}^2(0)+i(\bar{\rho}(0))-\bar{\Phi}(0)\\
    \Delta \Phi=w(x_1)\rho-b(x_1)
\end{cases}\quad\text{in $\mathcal{N}_L$,}
\end{equation}
with boundary conditions
\begin{align}
\label{elec-bc-1}
\rho{\bf u}\cdot {\bf n}_0=J+g_0,\quad \Phi=\bar{\Phi}(0)+h_0\quad&\mbox{on $\Gam_0$},\\
\label{elec-bc-2}
{\bf u}\cdot {\bf n}_w=0,\quad \nabla\Phi\cdot {\bf n}_w=0\quad&\mbox{on $\Gam_w$},\\
\label{elec-bc-3}
{\bf u}\cdot {\bm \tau}_L=0,\quad \der_{{\bf n}_L}\Phi=\bar{E}(L)+v_L\quad &\text{on $\Gam_L$}
\end{align}
where ${\bf n}_0$ and ${\bf n}_w$ are the inward unit normal vectors on $\Gamma_0$ and $\Gamma_w$, respectively, while ${\bf n}_L$ is the outward unit normal vector on $\Gamma_L$. Here, ${\bm\tau}_L$ denotes all tangential vector fields on $\Gamma_L$.

\end{problem}

\begin{theorem}[Euler--Poisson system of electric field type]
\label{theorem-1}
Let $n=2$ or $3$, and suppose that
\begin{equation}
\label{assumption1-thm1}
    \mu_0:=\min_{0\le x_1\le L} w(x_1)>0.
\end{equation}

\begin{itemize}
\item[(a)]
There exists a sufficiently small constant ${\sigma}_0>0$,
depending only on $(\bar{\rho}, \bar{u}_1, \bar{\Phi}, L, \mu_0)$,
such that, if the following conditions hold:
\begin{itemize}
\item[(i)]
\[
\|g_0\|_{C^2(\overline{\Gamma_0})}
+\|h_0\|_{C^3(\overline{\Gamma_0})}
+\|v_L\|_{C^2(\overline{\Gamma_L})}
\le {\sigma}_0;
\]

\item[(ii)]
\[
h_0'(x_2)=0
\quad \text{at } x_2=0,L;
\]

\item[(iii)]
\[
\|w'\|_{C^0([0,L])}\le {\sigma}_0,
\]
\end{itemize}
then Problem~\ref{main problem} admits a solution
\[
(\rho,{\bf u},\Phi)
\in [C^0(\ol{\mathcal{N}_L})\cap C^1(\mathcal{N}_L)]
\times [C^0(\ol{\mathcal{N}_L},\R^n)\cap C^1(\mathcal{N}_L, \R^n)]
\times [C^1(\ol{\mathcal{N}_L})\cap C^2(\mathcal{N}_L)],
\]
which is subsonic throughout $\overline{\mathcal{N}_L}$.

\item[(b)]
If, in addition, $p$ satisfies
\begin{equation}
\label{assumption2-thm1}
\frac{d}{d\rho}\left(\frac{\rho p''(\rho)}{p'(\rho)}\right)\le 0
\quad\text{for all } \rho>0,
\end{equation}
then the solution is unique.
\end{itemize}
\end{theorem}

\begin{remark}
Note that the pressure law $p(\rho)=\rho^{\gamma}$ for $\gamma\ge 1$,
which corresponds to an ideal polytropic gas,
satisfies assumption~\eqref{assumption2-thm1}.
\end{remark}

In contrast to Theorem \ref{theorem-1}, for the Euler–Poisson system of self gravitation type,
a global uniqueness result can also be established with a little more restrictive conditions.

\begin{theorem}[Euler-Poisson system of self-gravitation type]
\label{theorem-2}
Let $n=2$, and suppose that
\begin{equation*}
    \mu_1:=\max_{0\le x_1\le L} w(x_1)<0.
\end{equation*}
In addition, suppose that $p(\rho)$ is given by 
\[p(\rho)=\rho^{\gamma}\]
for a constant $\gamma\ge 1$. 
\begin{itemize}
\item[(a)] Then there exists a sufficiently small constant ${\sigma}_1>0$,
depending only on $(\bar{\rho}, \bar{u}_1, \bar{\Phi}, L, \mu_1)$,
such that if 
\begin{equation}
\label{assumption2-thm2}
\|g_0\|_{C^2(\overline{\Gamma_0})}
+\|h_0\|_{C^3(\overline{\Gamma_0})}
+\|v_L\|_{C^2(\overline{\Gamma_L})} +\|w'\|_{C^0([0,L])}
\le {\sigma}_1;
\end{equation}
and for $j=0,1,2$,
\begin{equation}
\label{assumption3-thm2}
\left(\frac{d}{dx_2}\right)^j g_0(x_2)
=\left(\frac{d}{dx_2}\right)^j h_0(x_2)
=\left(\frac{d}{dx_2}\right)^j v_L(x_2)
=0
\quad \text{at}\,\, x_2=0,L
\end{equation}
then Problem~\ref{main problem} admits a solution
\[
(\rho,{\bf u},\Phi)
\in C^1(\overline{\mathcal{N}_L})
\times C^1(\overline{\mathcal{N}_L},\mathbb{R}^2)
\times C^2(\overline{\mathcal{N}_L}),
\]
which is subsonic in $\overline{\mathcal{N}_L}$.

\item[(b)]
If, in addition,  $\gamma\ge 3$,
then the solution is unique.

\end{itemize}

\end{theorem}

There are a few remarks in order.

\begin{remark}
    The main contribution of this paper is to establish the global uniqueness of subsonic solutions without assuming that the solutions are close to a background state. 
In this paper, we consider linear boundary conditions instead of nonlinear ones, such as the pressure condition at the exit used in the works \cite{BDX_SIMA, bae2015two, bae20183, BDX_ARMA16}. The reason for this choice is that the boundary relation is then compatible with the energy estimates used in the uniqueness argument.
For nonlinear boundary conditions, the energy method does not directly yield a coercive estimate unless an additional monotonicity property is satisfied. In fact, for nonlinear elliptic problems, it is well known that global uniqueness under nonlinear boundary conditions typically requires a suitable monotonicity condition on the boundary operator. In the absence of such a condition, uniqueness may fail.
This suggests that the question of global uniqueness under nonlinear boundary conditions remains open.
In particular, for the Euler--Poisson system with pressure boundary conditions, it is not clear whether the associated boundary operator satisfies a suitable monotonicity property.
It would be interesting to investigate whether global uniqueness still holds in this setting or whether multiple subsonic solutions can arise.
\end{remark}

\begin{remark}
    The results of this paper provide a new framework for understanding global uniqueness in nonlinear elliptic systems arising from compressible fluid models with nonlocal interactions. In particular, our work shows that, despite the absence of a maximum principle, global uniqueness can still be achieved by exploiting intrinsic geometric structures of the solution set.
\end{remark}

\begin{remark}
   We note that the present paper focuses on the irrotational case. However, the approach developed here is expected to extend to flows with nonzero vorticity. In particular, the convexity-based framework and the associated energy estimates suggest a possible pathway toward establishing global uniqueness for more general flows with nonzero vorticity. This problem is currently under investigation and will be addressed in future work.
\end{remark}

\subsection{Key ideas for the proof of main results}
Our approach is based on identifying a hidden convex structure in the set of admissible subsonic states. In the case of flows under electric field, we introduce a $\delta$-subsonic set and show that it is convex. This convexity allows us to connect any two subsonic solutions by a continuous path that stays entirely within the subsonic regime. Using this observation together with energy-type estimates for the associated elliptic system, we show that the difference between two solutions must be zero. In the case of flows under self-gravitation field, we introduce a $\lambda$-set, analogous to the $\delta$-subsonic set, and prove that it is also convex. This shows that the same idea works in both cases.

A key novelty of this work is the unified treatment of flows with electric or self-gravitation field, which exhibit fundamentally different mathematical structures. For flows with electric field, the problem can be formulated in terms of a potential function, leading to a relatively direct definition of the $\delta$-subsonic set. 
In contrast, for flows with self-gravitation field, the analysis requires a different formulation based on a stream function. This formulation is crucial for establishing the existence of solutions, but it also leads to a substantially different structure of the problem. In particular, the notion of subsonicity must be reformulated in terms of the stream function, resulting in an analog of the $\delta$-subsonic set that differs from that with electric field. Due to this structural difference, the convexity of the corresponding set cannot be obtained under the same general assumptions, and additional conditions (such as a restriction on the adiabatic exponent) are required to ensure convexity.
This distinction indicates a fundamental difference between repulsive (electric) and attractive (gravitational) interactions in the Euler--Poisson system.

The rest of the paper is organized as follows. In Section \ref{section:2}, we study the uniqueness of subsonic solutions for the Euler-Poisson system of electric field type and  prove Theorem \ref{theorem-1}. The uniqueness of subsonic solutions of the Euler-Poisson system of self-gravitation type is investigated in Section \ref{section:3}, where Theorem \ref{theorem-2} is proved.

\newpage

\section{Proof of Theorem \ref{theorem-1}}\label{section:2}
In this section, we prove Theorem \ref{theorem-1}.
Assume that
\begin{equation}
\label{assumption1-thm1}
    \mu_0:=\min_{0\le x_1\le L} w(x_1)>0.
\end{equation}
Let $(\bar{\rho}, {\bar u}_1{\bf e}_1, \bar{\Phi})$ be a smooth subsonic solution of \eqref{EP-1d} on $[0,L]$ for a constant $J>0$. Put
\[
\mathcal{K}_0:=\frac 12 \bar{u}_1^2(0)+i(\bar{\rho}(0))-\bar{\Phi}(0).
\]
We refer to $\mathcal{K}_0$ as the pseudo-Bernoulli constant. We consider the irrotational flows, i.e.,  $\nabla\times {\bf u}=0$. Hence there exists a $\varphi$ such that ${\bf u}=\nabla\vphi.$
Therefore, one has
\begin{equation}
\label{pseudo-Bernoulli}
i(\rho)=\mathcal{K}_0+\Phi-\frac12 |\nabla\varphi|^2. 
\end{equation}
Since $\frac{d}{d\rho}i(\rho)=\frac{p'(\rho)}{\rho}>0$ for all $\rho>0$, the {\it{enthalpy function}} $i:(0, \infty)\rightarrow \{i(\rho):0<\rho<\infty\}$ is invertible, so we can write $\rho$ as
\[\rho=i^{-1}\left(\mathcal{K}_0+\Phi-\frac12 |\nabla\varphi|^2\right)=:\rho(\Phi, \nabla \vphi).\]
Therefore, we can finally reduce the system \eqref{EP-general-b} to
\begin{equation}
\label{EP-Potential}
    \begin{cases}
    {\rm div}\left(\rho(\Phi, \nabla\varphi)\nabla\varphi\right)=0\\
    \Delta\Phi=w(x_1)\rho(\Phi, \nabla\varphi)-b(x_1)
    \end{cases}\quad\text{in $\mathcal{N}_L$.}
\end{equation}
The flow is said to be subsonic (sonic, supersonic) if it satisfies
\[
p'(\rho(\Phi,\nabla\varphi))>(=, <)|\nabla\varphi|^2,
\]
respectively. 
Define
\begin{equation*}
\bar{\vphi}(x_1):=\int_0^{x_1}\bar u_1(t)\,dt.
\end{equation*}
We rewrite the boundary conditions \eqref{elec-bc-1}--\eqref{elec-bc-3} in terms of $(\vphi, \Phi)$ as follows:
\begin{equation}
\label{EP-potential-BC}
\begin{split}
\rho(\Phi, \nabla\vphi)\nabla\vphi\cdot{\bf n}_0=J+g_0,\quad \Phi=\bar{\Phi}(0)+h_0\quad&\mbox{on $\Gam_0$},\\
\nabla\vphi\cdot {\bf n}_w=0,\quad \nabla\Phi\cdot {\bf n}_w=0\quad&\mbox{on $\Gam_w$},\\
\vphi(L, x_2)=\bar{\vphi}(L),\quad \nabla\Phi\cdot {\bf n}_L=\bar{E}(L)+v_L\quad & \mbox{on $\Gam_L$}.
\end{split}
\end{equation}

The proof of Theorem~\ref{theorem-1} consists of two parts.
In \S\ref{subsection:existence 1}, we prove the existence of a subsonic solution to the boundary value problem
\eqref{EP-Potential}--\eqref{EP-potential-BC} by a straightforward modification of the argument in
\cite{BDX_ARMA16}, which establishes statement~(a) of Theorem~\ref{theorem-1}.
In \S\ref{subsection:uniqueness 1}, we introduce the {\it{$\delta$-subsonic set}} $\mathcal{P}_{\delta}$ for $\delta>0$
and show that it is convex.
Exploiting this convexity, we prove that any two subsonic solutions of Problem~\ref{main problem}
must coincide, regardless of their distance to the one-dimensional background solution.
This yields the global uniqueness result stated in part~(b) of Theorem~\ref{theorem-1}.

\subsection{The existence of a solution}
\label{subsection:existence 1}
Since $(\bar{\rho}, \bar{u}_1{\bf e}_1, \bar{\Phi})$ is a subsonic solution on $[0,L]$, we have
\begin{equation}\label{subsonicity}
\mathcal{K}_1 := \min_{0\le x_1\le L}p'(\rho(\bar{\Phi},\nabla\bar{\varphi}))-|\nabla\bar{\vphi}|^2>0.
\end{equation}
\begin{definition}
For a bounded connected open set $\Om\subset \R^n$, let $\Gam$ be a closed portion of $\der\Om$. For ${\rx}, {\rm y}\in \Om$, set
\[\delta_{\rx}:={\rm dist}({\rx, \Gam}), \quad \delta_{\rx, {\rm y}}:=\min (\delta_{\rx}, \delta_{\rm y}).\]
For $k\in \R$, $\alp\in(0,1)$, and $m\in \mathbb{Z}^+$, define the standard H\"{o}lder norms by
\begin{equation*}
    \begin{split}
    \|u\|_{m,0,\Om}&:=\sum_{0\le |\beta|\le  m}\sup_{\rx\in\Om}|D^{\beta}u({\rx})|,\quad 
[u]_{m,\alp,\Om}:=\sum_{|\beta|=m}\sup_{\rx, {\rm y}\in \Om,\atop {\rx}\neq {\rm y}} \frac{|D^{\beta}u(\rx)-D^{\beta}u({\rm y})|}{|\rx-{\rm y}|^{\alp}},\\
\|u\|_{m,\alp,\Om}&:=\|u\|_{m,0,\Om}+[u]_{m,\alp,\Om},
    \end{split}
\end{equation*}
and the weighted H\"{o}lder norms by
\begin{equation*}
    \begin{split}
    \|u\|_{m,0,\Om}^{(k,\Gam)}&:=\sum_{0\le |\beta|\le  m}\sup_{\rx\in\Om}\delta_{\rx}^{\max(|\beta|+k,0)}|D^{\beta}u({\rx})|,\\
[u]_{m,\alp,\Om}^{(k,\Gam)}&:=\sum_{|\beta|=m}\sup_{\rx, {\rm y}\in \Om,\atop {\rx}\neq {\rm y}} \delta_{\rx}^{\max(m+\alp+k,0)}\frac{|D^{\beta}u(\rx)-D^{\beta}u({\rm y})|}{|\rx-{\rm y}|^{\alp}},\\
\|u\|_{m,\alp,\Om}^{(k,\Gam)}&:=\|u\|_{m,0,\Om}+[u]_{m,\alp,\Om}^{(k,\Gam)},
    \end{split}
\end{equation*}
where $D^{\beta}$ denotes $\der_{x_1}^{\beta_1}\cdots \der_{x_n}^{\beta_n}$ for a multi-index $\beta=(\beta_1,\cdots, \beta_n)$ with $\beta_j\in \mathbb{Z}_+$ and $|\beta|=\sum_{j=1}^n$. The set $C^{m,\alp}_{(k,\Gam)}(\Om)$ denotes the completion of the set of all smooth functions whose $\|\cdot\|_{m,\alp, \Om}^{(k, \Gam)}$ norms are finite.
\end{definition}
Theorem~\ref{theorem-1}(a) follows directly from the proposition stated below.
\begin{prop}
\label{proposition-1}
Assume that \eqref{assumption1-thm1} holds.
Let $\alp\in(0,1)$ be given. There exists a constant $\sigma_0>0$ such that if the boundary data $(w,g_0, h_0, v_L)$ satisfy conditions {\rm (i)}--{\rm (iii)} in Theorem~\ref{theorem-1}, then the boundary value problem \eqref{EP-Potential}--\eqref{EP-potential-BC} admits a solution
$
(\vphi, \Phi)\in \bigl[C^{1,\alp}(\ol{\mathcal{N}_L})\cap C^{2,\alp}(\mathcal{N}_L)\bigr]^2,
$
which satisfies the following properties.
\begin{itemize}
\item[(a)] There exists a constant $\bar{\nu}>0$ such that
\[
p'(\rho(\Phi,\nabla\vphi)) - |\nabla\vphi|^2 \ge \bar{\nu}
\quad \text{in } \ol{\mathcal{N}_L}.
\]
\item[(b)] There exists a constant $C>0$ such that
\[
\|(\vphi, \Phi)-(\bar{\vphi}, \bar{\Phi})\|_{2,\alp,\mathcal{N}_L}^{(-1-\alp, \Gam_0\cup\Gam_L)}
\le
C\Bigl(
\|g_0\|_{C^2(\ol{\Gam_0})}
+\|h_0\|_{C^2(\ol{\Gam_0})}
+\|v_L\|_{C^2(\ol{\Gam_L})}
\Bigr).
\]
\end{itemize}
The constants $\sigma_0$, $\bar{\nu}$, and $C$ depend only on $(\bar{\rho}, \bar{u}_1, \bar{\Phi}, L, \mu_0, \alp)$.
\end{prop}

Proposition \ref{proposition-1} can be proved by an argument similar to that used in the proof of \cite[Theorem~1]{BDX_ARMA16}. In essence, the problem is solved via an iterative method. In this paper, we briefly explain the overall strategy for proving Proposition \ref{proposition-1} and present the key estimate. The remaining parts of the proof are omitted since they can be carried out by following the arguments already presented in \cite{BDX_ARMA16}.
\medskip

For $(z, {\bf q})=(z,q_1, \cdots, q_n)\in \R\times \R^n$, set
\begin{equation}
\label{definition:A,B}
    {\bf A}(z,{\bf q})=(A_1, \cdots, A_n)(z,{\bf q}):=\rho(z, {\bf q}){\bf q},\quad B(z, {\bf q}):=\rho(z, {\bf q}).
\end{equation}
Suppose that $(\vphi, \Phi)\in [C^1(\ol{\mathcal{N}_L})\cap C^2(\mathcal{N}_L)]^2$ solves the boundary value problem \eqref{EP-Potential}--\eqref{EP-potential-BC}, and set
\begin{equation*}
    (\phi, \Psi):=(\vphi, \Phi)-(\bar{\vphi}, \bar{\Phi}).
\end{equation*}
A direct computation yields the following boundary value problem for $(\phi, \Psi)$:
\begin{equation}
\label{equation-p}
 \begin{cases}
  L_1(\phi, \Psi)={\rm div}{\bf F}(x_1, \Psi, \nabla\phi)\\
  L_2(\phi, \Psi)=f(x_1, \Psi, \nabla\phi)
 \end{cases} \quad \mbox{in $\mathcal{N}_L$},
\end{equation}
with boundary conditions
\begin{align*}
 \begin{cases}
    \left( \der_{q_j}A_i(\bar{\Phi}, \nabla\bar{\vphi})\der_j\phi\right)\cdot {\bf n}_0=g_1\\
    \Psi=h_0
    \end{cases}\quad&\mbox{on $\Gam_0$},\\
 \nabla\phi\cdot {\bf n}_w=0,\quad \nabla\Psi\cdot {\bf n}_w=0\quad &\mbox{on $\Gam_w$},\\
 \phi(L, x_2)=0,\quad \nabla\Psi\cdot 
 {\bf n}_L=v_L\quad &\mbox{on $\Gam_L$}
\end{align*}
for
\begin{align}
\notag
L_1(\phi, \Psi)
    &:={\rm div}\left( \der_{q_j}A_i(\bar{\Phi}, \nabla\bar{\vphi})\der_j\phi+\Psi\der_z{\bf A}(\bar{\Phi}, \nabla\bar{\vphi})\right),\\
    \notag
L_2(\psi, \Psi)&:=\frac{1}{w(x_1)}\Delta \Psi-\der_zB(\bar{\Phi}, \nabla\bar{\vphi})\Psi-\der_{\bf q}B(\bar{\Phi}, \nabla\bar{\vphi})\cdot \nabla \phi,\\
\label{definition:F}
{\bf F}(x_1, z, {\bf q})&=(F_1, \cdots, F_n)(x_1, z, {\bf q})\quad\text{with}\\
\notag
-F_i(x_1, z, {\bf q})&=\int_0^1 z\left[\der_zA_i(\bar{\Phi}+\tilde z, \nabla\bar{\vphi}+\tilde{\bf q})\right]_{(\tilde z, \tilde{\bf q})=(0, {\bf 0})}^{(tz, {\bf q})}+q_j[\der_{q_j}A_i(\bar{\Phi}, \nabla\bar{\vphi}+\tilde{\bf q})]_{\tilde{\bf q}={\bf 0}}^{t{\bf q}}\,dt,\\
\label{definition:f}
f(x_1, z, {\bf q})&=\int_0^1 z\left[\der_z B(\bar{\Phi}+\tilde z, \nabla\bar{\vphi}+\tilde{\bf q})\right]_{(\tilde z, \tilde{\bf q})=(0, {\bf 0})}^{(tz, {\bf q})}+q_j[\der_{q_j}B(\bar{\Phi}, \nabla\bar{\vphi}+\tilde{\bf q})]_{\tilde{\bf q}={\bf 0}}^{t{\bf q}}\,dt,\\
\notag
g_1&:=g_0+{\bf F}(x_1, h_0, \nabla\phi)\cdot {\bf n}_0-h_0 {\bf n}_0\cdot \der_z{\bf A}(\bar{\Phi}, \nabla\bar{\vphi})
\end{align}
where $[G(X)]_{X=a}^b:=G(b)-G(a)$. Here, each $\der_j$ denotes $\der_{x_j}$ for $j=1,\cdots, n$.

A straightforward computation using \eqref{definition-enthalpy} and \eqref{pseudo-Bernoulli} yields
\begin{equation}
\label{coefficients-p}
    \der_{q_j}A_i(z,{\bf q})=\left(\delta_{ij}-\frac{q_iq_j}{p'\left(\rho(z, {\bf q})\right)}\right)\rho(z,{\bf q})
\end{equation}
provided that $\rho(z,{\bf q})>0$. In particular, for the background solution $(\bar{\Phi}, \bar{\vphi})$, this identity reduces to
\begin{equation*}
\frac{1}{\bar{\rho}}\der_{q_j}A_i(\bar{\Phi},\nabla\bar{\vphi})=
\begin{cases}
1-\frac{\bar u_1^2}{p'(\bar{\rho})}\quad&\mbox{if $i=j=1$},\\
1\quad&\mbox{if $i=j>1$},\\
0\quad&\mbox{if $i\neq j$}.
\end{cases}
\end{equation*}
By \eqref{subsonicity}, there exist positive constants $\lambda_0$ and $\Lambda_0$, depending only on the background solution, such that
\begin{equation}
\label{ellipticity-p}
\lambda_0 \mathbb{I}_n \le \begin{bmatrix}
\der_{q_j}A_i(\bar{\Phi},\nabla\bar{\vphi})
\end{bmatrix}_{i,j=1}^n\le \Lambda_0 \mathbb{I}_n \quad\text{in $\ol{\mathcal{N}_L}$.}
\end{equation}
Using this identity, we establish the well-posedness of the following linear boundary value problem:
\begin{equation}
\label{lbvp-1}
\begin{split}
 \begin{cases}
  L_1(U, V)={\rm div}{\bf F}\\
  L_2(U, V)=f
 \end{cases}\quad&\mbox{in $\mathcal{N}_L$},\\
 \begin{cases}
    \left( \der_{q_j}A_i(\bar{\Phi}, \nabla\bar{\vphi})\der_jU\right)\cdot {\bf n}_0=g\\
    V=h_1
    \end{cases}\quad&\mbox{on $\Gam_0$},\\
 \nabla U\cdot {\bf n}_w=0,\quad \nabla V\cdot {\bf n}_w=0\quad &\mbox{on $\Gam_w$},\\
 U(L, x_2)=0,\quad \nabla V\cdot 
 {\bf n}_L=h_2 \quad &\mbox{on $\Gam_L$}
 \end{split}
\end{equation}
where ${\bf F}$, $f$, $g$, $h_1$, and $h_2$ are given functions.

\begin{lemma}
\label{lemma-p-1}
Assume that the condition \eqref{assumption1-thm1} holds. Then, there exists a sufficiently small constant $\sigma_1>0$ depending only on $(\bar{\rho}, \bar u_1, \bar{\Phi}, L, \mu_0)$ such that if 
\begin{equation*}
\|w'\|_{C^0([0,L])}\le \sigma_1,
\end{equation*}
and if 
\begin{itemize}
\item[(i)] ${\bf F}, f\in C^{1,\alp}_{(-\alp, \Gam_0\cup\Gam_L)}(\mathcal{N}_L) $;
\item[(ii)] $g\in C^{1,\alp}(\ol{\Gam_0})$, $h_1\in C^{2,\alp}(\ol{\Gam_0})$, $h_2\in C^{1,\alp}(\ol{\Gam_L})$;
\item[(iii)] $h_1'(0)=h_1'(L)=0$,
\end{itemize}
then the linear boundary value problem \eqref{lbvp-1} admits a unique solution $(U,V)\in C^{2,\alp}_{(-1-\alp, \Gam_0\cup\Gam_L)}(\mathcal{N}_L)$. Furthermore, the solution satisfies the estimate
\begin{equation*}
\begin{split}
&\|U\|_{2,\alp,\mathcal{N}_L}^{(-1-\alp, \Gam_0\cup\Gam_L)}+\|V\|_{2,\alp,\mathcal{N}_L}^{(-1-\alp, \Gam_0\cup\Gam_L)}\\
&\le C
\left(\|{\bf F}\|_{1,\alp,\mathcal{N}_L}^{(-\alp, \Gam_0\cup\Gam_L)}+\|f\|_{1,\alp,\mathcal{N}_L}^{(-\alp, \Gam_0\cup\Gam_L)}+\|g\|_{1,\alp,\Gam_0}+\|h_1\|_{2,\alp, \Gam_0}+\|h_2\|_{1,\alp, \Gam_L}\right)
\end{split}
\end{equation*}
for a constant $C>0$ fixed depending only on $(\bar{\rho}, \bar u_1, \bar{\Phi}, L, \mu_0)$ and $\|w\|_{C^3([0,L])}$. 
\end{lemma}

Lemma \ref{lemma-p-1} can be proved in three steps:
\begin{itemize}
\item[Step 1.] {\it{Establish a priori $H^1$-estimate of $(U,V)$ to apply the Lax-Milgram theorem, from which the unique existence of a weak solution to \eqref{lbvp-1} follows.}} 

\item[Step 2.] {\it{Establish a priori $C^{\alp}$-estimate of the weak solution $(U, V)$ in $\mathcal{N}_L$.}}

\item[Step 3.] {\it{By bootstrap arguments, establish a priori estimates of $\|U\|_{2,\alp, \mathcal{N}_L}^{(-1-\alp, \Gam_0\cup\Gam_L)}$ and $\|V\|_{2,\alp, \mathcal{N}_L}^{(-1-\alp, \Gam_0\cup\Gam_L)}$. This shows that $(U,V)$ is a classical solution to \eqref{lbvp-1}.}}
\end{itemize}
\medskip

Steps 2 and 3 can be carried out by following the arguments in \cite{BDX_ARMA16}; therefore, we omit the details here. 
The main difference between the system considered in this paper and that in \cite{BDX_ARMA16} lies in the fact that the function $w(x_1)$ is no longer constant in this paper. 
This difference affects the derivation of the {\it{a priori} $H^{1}$-estimate}, and hence we present the corresponding estimate in detail in this paper.

Without loss of generality, we may assume that $h_1=0$, which can be achieved by subtracting $h_1$ from $V$. Define
\begin{align*}
H^1_{\Gamma_L}(\mathcal{N}_L)&:=\{\xi_1\in H^1(\mathcal{N}_L)\vert\,\, \xi_1=0\,\,\text{on}\,\,\Gamma_L\,\,\text{in the trace sense}\}\\
H^1_{\Gamma_0}(\mathcal{N}_L)&:=\{ \xi_2 \in H^1(\mathcal{N}_L)\vert\,\, \xi_2=0\,\,\text{on}\,\,\Gamma_0\,\,\text{in the trace sense}\}.
\end{align*}
\begin{definition}
\label{definition-p-weak} 
$(U, V)\in H^1_{\Gamma_L}(\mathcal{N}_L)\times H^1_{\Gamma_0}(\mathcal{N}_L)$ is said to be a weak solution to \eqref{lbvp-1} with $h_1=0$ on $\Gam_0$ if it satisfies
\begin{equation*}
\begin{split}
\mathcal{B}_1[(U,V),\xi_1]
    &=\int_{\mathcal{N}_L}{\bf F}\cdot \nabla\xi_1\,d\rx-\int_{\Gam_0}g\xi_1\,dx_2-\int_{\Gam_0\cup\Gam_w} ({\bf F}\cdot {\bf n}_{\rm out})\xi_1\,dS_{\rx}\,\,\text{for all}\,\,\xi_1\in H^1_{\Gamma_L}(\mathcal{N}_L)
    \end{split}
\end{equation*}
and 
\begin{equation*}
    \begin{split}
&\mathcal{B}_2[(U,V),\xi_2]
    =-\int_{\mathcal{N}_L}f\xi_2\,d\rx+\frac{1}{w(L)}\int_{\Gam_L} h_2\xi_2\,dx_2 \,\,\text{for all}\,\,\xi_2\in H^1_{\Gamma_0}(\mathcal{N}_L),
    \end{split}
\end{equation*}
where
\begin{equation}
\mathcal{B}_1[(U,V),\xi_1]:=\int_{\mathcal{N}_L} \left(\der_{q_j}A_i(\bar{\Phi}, \nabla\bar{\vphi})\der_j U+V\der_z{\bf A}(\bar{\Phi}, \nabla\bar{\vphi}) \right )\cdot \nabla \xi_1\,d\rx, 
\end{equation}
and
\begin{equation*}
\begin{aligned}
&\mathcal{B}_2[(U,V),\xi_2]\\
:=& \int_{\mathcal{N}_L} \frac{1}{w(x_1)} \nabla V\cdot \nabla \xi_2
    +\left(\der_zB(\bar{\Phi}, \nabla\bar{\vphi})V+\der_{\bf q}B(\bar{\Phi}, \nabla\bar{\vphi})\cdot \nabla U-\frac{w'(x_1)}{w^2(x_1)}\der_1 V\right)\xi_2\,d\rx.
    \end{aligned}
    \end{equation*}
\end{definition}

\begin{lemma}
Assume that the condition \eqref{assumption1-thm1} holds. Then, there exists a sufficiently small constant $\sigma_1>0$ depending only on $L$ and $\mu_0$ such that if 
\begin{equation*}
\|w'\|_{C^0([0,L])}\le \sigma_1,
\end{equation*}
and if $(U,V)$ is a weak solution to \eqref{lbvp-1} with $h_1=0$ on $\Gam_0$ in the sense of Definition \ref{definition-p-weak}, then we have
\begin{equation}
\label{H1-estimate-p}
\begin{split}
&\|U\|_{H^1(\mathcal{N}_L)}
+\|V\|_{H^1(\mathcal{N}_L)}\\
\le & C\left(\|{\bf F}\|_{L^2(\mathcal{N}_L)}+\|{\bf F}\|_{L^2(\der\mathcal{N}_L)}+\|{f}\|_{L^2(\mathcal{N}_L)}+\|g\|_{L^2(\Gam_0)}+\|h_2\|_{L^2(\Gam_L)}\right)
\end{split}
\end{equation}
for some constant $C>0$ fixed depending only on $(\bar{\rho}, \bar u_1, \bar{\Phi}, L, \mu_0)$. 

\begin{proof}
First of all, we make a very simple but essential observation. By differentiating ${\bf A}(z,{\bf q})$ and $B(z,{\bf q})$ given by \eqref{definition:A,B} with respect to $z$ and ${\bf q}$, respectively, using \eqref{definition-enthalpy} and \eqref{pseudo-Bernoulli}, we obtain
\begin{equation}
\label{derivatives-A-B}
    \der_z{\bf A}(z,{\bf q})=\frac{\rho}{p'(\rho)}{\bf q},\quad \der_{\bf q}B(z,{\bf q})=-\frac{\rho}{p'(\rho)}{\bf q},\quad \der_zB=\frac{\rho}{p'(\rho)}.
\end{equation}
wherever $\rho(z,{\bf q})>0$. This implies
\begin{equation}
\label{essential observation-p}
\der_z{\bf A}(z,{\bf q})+\der_{\bf q}B(z,{\bf q})=0.
\end{equation}

Therefore, we obtain
\begin{equation*}
\begin{aligned}
&\mathcal{B}_1[(U,V),U]+\mathcal{B}_2[(U,V),V]\\
= &\int_{\mathcal{N}_L}
\der_{q_j}A_i(\bar{\Phi}, \nabla\bar{\vphi})\,\der_i U\,\der_j U
+\frac{1}{w(x_1)}|\nabla V|^2 
+\der_zB(\bar{\Phi}, \nabla\bar{\vphi})\,V^2
-\frac{w'}{w^2}V\,\der_1 V \, d\rx .
\end{aligned}
\end{equation*}

By \eqref{assumption1-thm1}, the weight function $w$ satisfies
\begin{equation}
\label{ragne of w-p}
0<\mu_0\le w(x_1)\le \mu_0+L\|w'\|_{C^0([0,L])}
\quad\text{for all } x_1\in[0,L].
\end{equation}
Combining this bound with the ellipticity condition \eqref{ellipticity-p}, we deduce that
\begin{equation*}
\begin{aligned}
&\int_{\mathcal{N}_L}
\der_{q_j}A_i(\bar{\Phi}, \nabla\bar{\vphi})\,\der_i U\,\der_j U
+\frac{1}{w(x_1)}|\nabla V|^2\, d\rx\\
\ge &
\int_{\mathcal{N}_L}
\lambda_0|\nabla U|^2
+\frac{1}{\mu_0+L\|w'\|_{C^0([0,L])}}|\nabla V|^2\, d\rx .
\end{aligned}
\end{equation*}
Also, it follows from \eqref{derivatives-A-B} that
\begin{equation*}
\int_{\mathcal{N}_L} \der_z B(\bar{\Phi}, \nabla\bar{\vphi})V^2\,d\rx=\int_{\mathcal{N}_L} \frac{\bar{\rho}}{p'(\bar{\rho})}V^2\,d\rx\ge 0.
\end{equation*}
Using $V=0$ on $\Gamma_0$ and the fundamental theorem of calculus, we obtain
\begin{equation*}
\int_{\mathcal{N}_L} V^2\,d\rx
\le \frac{L^2}{2}\int_{\mathcal{N}_L}|\der_1 V|^2\,d\rx ,
\end{equation*}
which extends to $H^1_{\Gamma_0}(\mathcal{N}_L)$. Consequently,
\begin{equation*}
\begin{split}
\int_{\mathcal{N}_L} \left|\frac{w'
}{w^2} V\der_1V\right|\,d\rx &\le\frac{\|w'\|_{C^0([0,L])}}{\mu_0^2}\|V\|_{L^2(\mathcal{N}_L)} \|\der_1 V\|_{L^2(\mathcal{N}_L)}\\
&\le \frac{L\|w'\|_{C^0([0,L])}}{\sqrt{2}\mu_0^2}\int_{\mathcal{N}_L} |\nabla V|^2\,d\rx.
\end{split}
\end{equation*}
This yields
\begin{equation*}
\begin{split}
&\mathcal{B}_1[(U,V),U]+\mathcal{B}_2[(U,V),V]\\
\ge &  \int_{\mathcal{N}_L}\lambda_0|\nabla U|^2+\left(\frac{1}{\mu_0+L\|w'\|_{C^0([0,L])}}-\frac{L\|w'\|_{C^0([0,L])}}{\sqrt{2}\mu_0^2}\right)|\nabla V|^2\,d\rx.
\end{split}
\end{equation*}
From this estimate, we conclude that if $\|w'\|_{C^0([0,L])}$ is sufficiently small,
depending only on $\mu_0$ and $L$, then the bilinear form
$\mathcal{B}_1+\mathcal{B}_2$ is coercive. More precisely, we can fix a constant $\sigma_1>0$ sufficiently small depending only on $\mu_0$ and $L$ such that if 
\[\|w'\|_{C^0([0,L])}\le \sigma_0,\]
then 
\begin{equation*}
\mathcal{B}_1[(U,V),U]+\mathcal{B}_2[(U,V),V] \ge \int_{\mathcal{N}_L} \lambda_0|\nabla U|^2+\frac{1}{2\mu_0}|\nabla V|^2\,d\rx.
\end{equation*}

Finally, the estimate \eqref{H1-estimate-p} follows directly from the above inequality and the Cauchy--Schwarz inequality.
\end{proof}

\end{lemma}

We now give a brief outline of the proof of Proposition \ref{proposition-1}, based on Lemma \ref{lemma-p-1}.
It follows from \eqref{definition:A,B} that there exists a sufficiently small constant
$\bar{\delta}>0$, depending only on $(\bar{\rho}, \bar{u}_1, \bar{\Phi}, L)$, such that for each
$x_1\in[0,L]$, the functions ${\bf F}(x_1, z, {\bf q})$ and $f(x_1, z, {\bf q})$, defined in
\eqref{definition:F} and \eqref{definition:f}, respectively, are smooth with respect to $(z,{\bf q})$
in the region
\begin{equation}
\label{zq-domain}
|z|+|{\bf q}|\le \bar{\delta}.
\end{equation}
Moreover, one can directly verify that
\begin{equation}
\label{contraction}
\begin{split}
&|{\bf F}(x_1, z, {\bf q})-{\bf F}(x_1, z', {\bf q}')|
 + |f(x_1, z, {\bf q})-f(x_1, z', {\bf q}')| \\
 \le &
C\,(|z|+|z'|+|{\bf q}|+|{\bf q}'|)\, |(z, {\bf q})-(z', {\bf q}')|
\end{split}
\end{equation}
for all $(z, {\bf q})$ and $(z', {\bf q}')$ satisfying \eqref{zq-domain}, where $C>0$ is a constant
depending only on $(\bar{\rho}, \bar{u}_1, \bar{\Phi}, L)$.

Suppose that $(\tilde{\Psi}_j, \tilde{\phi}_j)$, $j=1,2$, satisfy
\[
\|\tilde{\Psi}_j\|_{2,\alp, \mathcal{N}_L}^{(-1-\alp, \Gam_0\cup\Gam_L)}
+\|\tilde{\phi}_j\|_{2,\alp, \mathcal{N}_L}^{(-1-\alp, \Gam_0\cup\Gam_L)}
\le \delta
\]
for some $\delta\in(0,\bar{\delta}]$. For each $j=1,2$, define
\[
{\bf F}_j:={\bf F}(x_1, \tilde{\Psi}_j, \nabla\tilde{\phi}_j),
\qquad
f_j:=f(x_1, \tilde{\Psi}_j, \nabla\tilde{\phi}_j).
\]
Then, by \eqref{contraction}, one can directly check that
\begin{equation*}
\begin{split}
&\|{\bf F}_1-{\bf F}_2\|_{1,\alp,\mathcal{N}_L}^{(-\alp, \Gam_0\cup\Gam_L)}
+\|f_1-f_2\|_{1,\alp,\mathcal{N}_L}^{(-\alp, \Gam_0\cup\Gam_L)} \\
 \le &
C_*\delta\,
\|(\tilde{\Psi}_1,\tilde{\phi}_1)-(\tilde{\Psi}_2,\tilde{\phi}_2)\|_{2,\alp,\mathcal{N}_L}^{(-1-\alp, \Gam_0\cup\Gam_L)}
\end{split}
\end{equation*}
for some constant $C_*>0$ depending only on $(\bar{\rho}, \bar{u}_1, \bar{\Phi}, L)$.
Therefore, Proposition \ref{proposition-1} follows by applying Lemma \ref{lemma-p-1}
together with the contraction mapping principle.

\subsection{The uniqueness}
\label{subsection:uniqueness 1}

We now turn to the proof of statement (b) of Theorem \ref{theorem-1}. The key ingredient in proving this statement is the following proposition.

\begin{prop}[Convexity of the {\it{$\delta$-subsonic set}} $\mathcal{P}_{\delta}$ ]
\label{proposition-convexity-p}
For $\delta \in \R$, define
\[
\mathcal{P}_{\delta}
:=\{(z,{\bf q})\in \R\times \R^n : p'(\rho(z,{\bf q}))-|{\bf q}|^2 \ge \delta\}\quad (n=2,3).
\]
Under the assumption \eqref{assumption2-thm1}, the set $\mathcal{P}_{\delta}$ is convex for any $\delta >0$.

\end{prop}
\begin{proof} 
This proof is inspired by the argument introduced in \cite{Bonheure2016PathsTU}.

If $\mathcal{P}_{\delta} = \emptyset$, there is nothing to prove. 
Suppose that $(z_j, {\bf q}_j)\in \mathcal{P}_{\delta}$ for $j=0,1$, and put
\begin{equation*}
(z_t, {\bf q}_t):=(1-t)(z_0, {\bf q}_0)+t(z_1, {\bf q}_1).
\end{equation*}
For $t\in[0,1]$, define
\begin{equation*}
\begin{split}
\rho_t &:=\rho(z_t, {\bf q}_t),\\
\mathfrak{s}(t) &:= p'(\rho_t)-|{\bf q}_t|^2.
\end{split}
\end{equation*}
Since $p'(\rho_j)\ge \delta>0$ for $j=0,1$, it follows that $\rho_j>0$, and hence
\[
i(\rho_j)>\lim_{\rho\to 0+} i(\rho), \qquad j=0,1,
\]
where $i(\rho)$ is defined in \eqref{pseudo-Bernoulli}. Moreover, since
\[
i(\rho_t)=\mathcal{K}_0+(1-t)z_0+tz_1-\frac{1}{2}\big|(1-t){\bf q}_0+t{\bf q}_1\big|^2,
\]
we compute
\[
\frac{d^2}{dt^2} i(\rho_t) = -|{\bf q}_1-{\bf q}_0|^2 \le 0 \qquad \text{for all } t\in\mathbb{R}.
\]
Therefore,
\[
i(\rho_t)\ge \min\{i(\rho_0), i(\rho_1)\} > \lim_{\rho\to 0+} i(\rho),
\]
which implies that $\rho_t$ is well-defined for all $t\in[0,1]$ and that
\begin{equation}
\label{convexity-density-p}
    \rho_t \ge \min\{\rho_0, \rho_1\}
\end{equation}
since $i(\rho_t)$ is concave in $t\in[0,1]$.

It follows from \eqref{definition-enthalpy} that
\[
\frac{d\rho_t}{dt}=\frac{\rho_t}{p'(\rho_t)}\frac{di(\rho_t)}{dt}.
\]
This yields
\begin{equation*}
\begin{split}
\frac{d^2}{dt^2}p'(\rho_t)&=\frac{d}{dt}\left(\frac{p''(\rho_t)\rho_t}{p'(\rho_t)}\frac{di(\rho_t)}{dt}\right)\\
&=\frac{p''(\rho_t)\rho_t}{p'(\rho_t)}\frac{d^2}{dt^2}i(\rho_t)+\frac{d}{d\rho_t}\left(\frac{p''(\rho_t)\rho_t}{p'(\rho_t)}\right)\left(\frac{di(\rho_t)}{dt}\right)^2\frac{\rho_t}{p'(\rho_t)}\\
&=-\frac{p''(\rho_t)\rho_t}{p'(\rho_t)}|{\bf q}_1-{\bf q}_0|^2+\frac{d}{d\rho_t}\left(\frac{p''(\rho_t)\rho_t}{p'(\rho_t)}\right)\left(\frac{di(\rho_t)}{dt}\right)^2\frac{\rho_t}{p'(\rho_t)}.
\end{split}
\end{equation*}
Therefore, we obtain
\begin{equation*}
    \frac{d^2}{dt^2}\mathfrak{s}(t)
    = -\left(1+\frac{p''(\rho_t)\rho_t}{p'(\rho_t)}\right)\lvert {\bf q}_1-{\bf q}_0\rvert^2
    + \frac{d}{d\rho_t}\!\left(\frac{p''(\rho_t)\rho_t}{p'(\rho_t)}\right)
    \left(\frac{di(\rho_t)}{dt}\right)^2\frac{\rho_t}{p'(\rho_t)}.
\end{equation*}
Under the assumption \eqref{assumption2-thm1}, we have
\begin{equation}
\label{convextiy-s}
\frac{d^2}{dt^2}\mathfrak{s}(t)\le 0
\quad\text{for all } t\in[0,1].
\end{equation}
This indicates that $\mathfrak{s}(t)$ is concave on $t\in[0,1]$.
Since $\mathfrak{s}(0)\ge \delta$ and $\mathfrak{s}(1)\ge \delta$, it follows from the concavity of $\mathfrak{s}(t)$ that
\[
\mathfrak{s}(t)\ge \delta \quad\text{for all } t\in[0,1].
\]
Therefore, $(z_t, {\bf q}_t)\in \mathcal{P}_{\delta}$ for all $t\in[0,1]$. This completes the proof.

\end{proof}

\begin{proof}[Proof of part (b) of Theorem \ref{theorem-1}] The proof is divided into three steps.

{{\it Step 1.}} For $j=0,1$, let
\[
(\rho_j, {\bf u}_j, \Phi_j)\in
\bigl[C^0(\ol{\mathcal{N}_L})\cap C^1(\mathcal{N}_L)\bigr]
\times
\bigl[C^0(\ol{\mathcal{N}_L};\R^n)\cap C^1(\mathcal{N}_L;\R^n)\bigr]
\times
\bigl[C^1(\ol{\mathcal{N}_L})\cap C^2(\mathcal{N}_L)\bigr]
\]
be two subsonic solutions of Problem~\ref{main problem} in $\ol{\mathcal{N}_L}$. One has
\[
\delta_j:=\min_{\ol{\mathcal{N}_L}}\bigl(p'(\rho_j)-|{\bf u}_j|^2\bigr)>0.
\]

Define 
\begin{equation}
\label{definition-deltas}
    \delta_*:=\min\{\delta_0, \delta_1\}.
\end{equation}

For each $j=0,1$, consider the boundary value problem
\begin{equation}
\label{aux-bvp-p}
\begin{split}
\Delta \varphi = \nabla\cdot {\bf u}_j \quad &\text{in } \mathcal{N}_L,\\
\nabla \varphi \cdot {\bf n} = {\bf u}_j \cdot {\bf n} \quad 
&\text{on } \partial\mathcal{N}_L \setminus \Gamma_L,\\
\varphi = \bar{\varphi}(L) \quad &\text{on } \Gamma_L,
\end{split}
\end{equation}
where ${\bf n}$ denotes the outward unit normal vector on $\partial\mathcal{N}_L$.
For each $j=0,1$, it is well known that the boundary value problem \eqref{aux-bvp-p}
admits a unique weak solution $\varphi_j \in H^1(\mathcal{N}_L)$.
Moreover, $\varphi_j$ belongs to $C^1(\ol{\mathcal{N}_L})\cap C^2(\mathcal{N}_L)$.

Define
\[
{\bf w} := \nabla \varphi_0 - {\bf u}_0.
\]
Then
\[
\nabla \times {\bf w} = {\bf 0} \quad \text{in } \mathcal{N}_L,
\]
and hence there exists a function $\phi \in C^2(\mathcal{N}_L)$ such that
${\bf w} = \nabla \phi$.
Furthermore, it follows from \eqref{aux-bvp-p} that $\phi$ satisfies
\[
\begin{cases}
\Delta \phi = 0 & \text{in } \mathcal{N}_L,\\
\nabla \phi \cdot {\bf n} = 0 & \text{on } \partial\mathcal{N}_L \setminus \Gamma_L,\\
\phi = \text{constant} & \text{on } \Gamma_L,
\end{cases}
\]
where the last condition follows from $\partial_{\bm\tau}\phi = 0$ on $\Gamma_L$
for all unit tangential vectors ${\bm\tau}$ on $\Gamma_L$. This implies that $\phi$ is constant in $\overline{\mathcal{N}_L}$, and hence
\[
\nabla \varphi_0 = {\bf u}_0 \quad \text{in } \overline{\mathcal{N}_L}.
\]
By the same argument, we also obtain
\[
\nabla \varphi_1 = {\bf u}_1 \quad \text{in } \overline{\mathcal{N}_L}.
\]
Consequently, $(\varphi_0, \Phi_0)$ and $(\varphi_1, \Phi_1)$ both solve the nonlinear boundary value problem
\eqref{EP-Potential}--\eqref{EP-potential-BC}.
\medskip

{\it Step 2.} For $t\in[0,1]$, define
\[
(\varphi_t, \Phi_t):=(1-t)(\varphi_0, \Phi_0)+t(\varphi_1, \Phi_1).
\]
For $\der_{q_j}A_i(z,{\bf q})$ given by \eqref{coefficients-p}, define
\[
a_{ij}^t:=\der_{q_j}A_i(\Phi_t,\nabla\varphi_t)
\]
for $t\in[0,1]$. 

By Proposition \ref{proposition-convexity-p} and \eqref{definition-deltas}, we have
\begin{equation}
\label{matrix-prop1-p}
\min_{\ol{\mathcal{N}_L}} \left(p'(\rho_t)-|\nabla\varphi_t|^2\right) \ge \delta_*\quad\text{for all $t\in[0,1]$.}
\end{equation}

By \eqref{definition-deltas}, we have
\[
\min_{j=0,1} p'(\rho_j(\rx)) \ge \delta_*
\quad\text{for all $\rx\in \ol{\mathcal{N}_L}$.}\]
Since $p''(\rho)>0$, the function $p'$ is strictly increasing. Hence, there exists a constant ${\varrho}_*>0$, depending only on $\delta_*$, such that
\[
\min_{j=0,1} \rho_j(\rx) \ge {\varrho}_* \quad\text{for all $\rx\in \ol{\mathcal{N}_L}$}.
\]
Combining this with \eqref{convexity-density-p}, we deduce that
 \begin{equation}
\label{matrix-prop2-p}
\rho_t\ge {\varrho}_* \quad\text{for all $\rx\in \ol{\mathcal{N}_L}$}.
 \end{equation}
Set
\[
\mathcal{K}_0^* := \mathcal{K}_0 + \max_{j=0,1} \|\Phi_j\|_{C^0(\ol{\mathcal{N}_L})}.
\]

By \eqref{pseudo-Bernoulli}, it follows that
\[
i(\rho_t) \le \mathcal{K}_0^* \quad \text{in } \ol{\mathcal{N}_L} \text{ for all } t \in [0,1].
\]
Therefore, there exists a constant $\varrho^*>0$, depending only on $\mathcal{K}_0^*$, such that
\begin{equation}
\label{matrix-prop3-p}
\rho_t \le \varrho^* \quad \text{in } \ol{\mathcal{N}_L} \text{ for all } t \in [0,1].
\end{equation}

As a matrix, ${\mathbb A}^t:=[a_{ij}^t]_{i,j=1}^n$ is symmetric. For each ${\bm \xi}=(\xi_1,\cdots, \xi_n)\in \R^n$ for $n=2$ or 3, observe 
\begin{equation}
\label{def-lambda}
\begin{split}
{\bm\xi}{\mathbb A}^t{\bm\xi}^T
= \left(|{\bm\xi}|^2-\frac{(\nabla\varphi_t\cdot {\bm\xi})^2}{p'(\rho_t)}\right)\rho_t\ge \left(1-\frac{|\nabla\varphi_t|^2}{p'(\rho_t)}\right)\rho_t |{\bm\xi}|^2\quad\text{in $\ol{\mathcal{N}_L}$.}
\end{split}
\end{equation}
Combining \eqref{def-lambda} with \eqref{matrix-prop1-p}--\eqref{matrix-prop3-p}, we obtain that
\begin{equation}
\label{ellipticity-p-uq}
\frac{\delta_*}{p'(\varrho^*)} |{\bm\xi}|^2 \le {\bm\xi}{\mathbb A}^t{\bm\xi}^T \le \varrho^* |{\bm\xi}|^2 \quad\text{in $\ol{\mathcal{N}_L}$ for all $\bm\xi\in \R^n$, $t\in [0,1]$. }
\end{equation}

{{\it Step 3.}} 
Since both $(\varphi_0,\Phi_0)$ and $(\varphi_1, \Phi_1)$ are classical solutions to \eqref{EP-Potential}--\eqref{EP-potential-BC}, we have
\begin{equation*}
\begin{cases}
\int_{\mathcal{N}_L} (\vphi_1-\vphi_0) \nabla\cdot\left({\bf A}(\Phi_1, \nabla\vphi_1)-{\bf A}(\Phi_0, \nabla\vphi_0)\right)\,d\rx=0\\
\int_{\mathcal{N}_L}(\Phi_1-\Phi_0)\left(\frac{1}{w(x_1)}\Delta(\Phi_1-\Phi_0)-\left(\rho(\Phi_1, \nabla\vphi_1)-\rho(\Phi_0, \nabla\vphi_0)\right)\right)\,d\rx=0
\end{cases}
\end{equation*}
for ${\bf A}(z,{\bf q})$ given by \eqref{definition:A,B}.
Furthermore, each $\rho_j=\rho(\Phi_j,\nabla\varphi_j)$ belongs to $C^0(\ol{\mathcal{N}_L})$, which implies that $(\varphi_j,\Phi_j)$ is a weak solution of \eqref{EP-Potential}--\eqref{EP-potential-BC}. Since $\partial\mathcal{N}_L$ is Lipschitz, the integration by parts formula is therefore applicable, and we obtain the following.

\begin{equation}
\label{energy-1}
    \begin{split}
    0&=\int_{\mathcal{N}_L}(\Phi_1-\Phi_0)\left(\frac{1}{w(x_1)}\Delta(\Phi_1-\Phi_0)-(\rho(\Phi_1, \nabla\vphi_1)-\rho(\Phi_0, \nabla\vphi_0)\right)\,d\rx\\
    &=-\int_{\mathcal{N}_L}\frac{1}{w(x_1)}|\nabla(\Phi_1-\Phi_0)|^2+(\Phi_1-\Phi_0)^2\int_0^1\rho_z(\Phi_t, \nabla\varphi_t)\,dt \,d\rx\\
    &\phantom{.......}-\int_{\mathcal{N}_L}(\Phi_1-\Phi_0) \nabla(\vphi_1-\vphi_0)\cdot \int_0^1  \rho_{\bf q}(\Phi_t, \nabla\varphi_t)\,dt\,d\rx\\
&\phantom{.......}+\int_{\mathcal{N}_L} \frac{w'(x_1)}{w^2(x_1)}(\Phi_1-\Phi_0)\der_1(\Phi_1-\Phi_0)\,d\rx;
     \end{split}
\end{equation}

\begin{equation}
\label{energy-2}
\begin{split}
   0 &=\int_{\mathcal{N}_L} (\vphi_1-\vphi_0) \nabla\cdot\left({\bf A}(\Phi_1, \nabla\vphi_1)-{\bf A}(\Phi_0, \nabla\vphi_0)\right)\,d\rx\\
&=-\int_{\mathcal{N}_L} \left({\bf A}(\Phi_1, \nabla\vphi_1)-{\bf A}(\Phi_0, \nabla\vphi_0)\right)\cdot \nabla(\vphi_1-\vphi_0)\,d\rx\\
&=-\int_{\mathcal{N}_L} \nabla(\vphi_1-\vphi_0)\cdot (\Phi_1-\Phi_0) \int_0^1{\bf A}_z(\Phi_t, \nabla\vphi_t)\,dt\,d\rx\\
&\phantom{.......}-\int_{\mathcal{N}_L} \der_i(\vphi_1-\vphi_0)\der_j (\vphi_1-\vphi_0)\int_0^1\der_{q_i}A_j (\Phi_t, \nabla\vphi_t)\,dt\,d\rx.
\end{split}
\end{equation}

By \eqref{ellipticity-p-uq}, we obtain
\begin{equation*}
\int_{\mathcal{N}_L} 
\der_i(\vphi_1-\vphi_0)\der_j(\vphi_1-\vphi_0)
\int_0^1 \der_{q_i}A_j(\Phi_t, \nabla\vphi_t)\,dt\, d\rx
\ge 
\frac{\delta_*}{p'(\varrho^*)}
\int_{\mathcal{N}_L} |\nabla(\varphi_1-\varphi_0)|^2 \, d\rx.
\end{equation*}

Moreover, it follows from \eqref{essential observation-p} that
\begin{equation*}
\rho_{\bf q}(\Phi_t, \nabla\varphi_t)
+ {\bf A}_z(\Phi_t, \nabla\varphi_t)
= 0
\quad \text{in } \ol{\mathcal{N}_L}
\quad \text{for all } t \in [0,1].
\end{equation*}

Adding \eqref{energy-1}--\eqref{energy-2} and applying \eqref{ragne of w-p}, we deduce that
\begin{equation*}
\begin{split}
&\int_{\mathcal{N}_L}
\frac{1}{\mu_0+L\|w'\|_{C^0([0,L])}}|\nabla(\Phi_1-\Phi_0)|^2
+ \frac{\delta_*}{p'(\varrho^*)}
|\nabla(\varphi_1-\varphi_0)|^2
 d\rx\\
 &+\int_{\mathcal{N}_L}
(\Phi_1-\Phi_0)^2
\int_0^1 \rho_z(\Phi_t, \nabla\varphi_t)\,dt\, d\rx\\
\le & \int_{\mathcal{N}_L} \frac{w'(x_1)}{w^2(x_1)}(\Phi_1-\Phi_0)\der_1(\Phi_1-\Phi_0)\,d\rx.
\end{split}
\end{equation*}
By \eqref{derivatives-A-B} and \eqref{matrix-prop2-p}, we have
\[\min_{\ol{\mathcal{N}_L}}\rho_z(\Phi_t,\nabla\varphi_t)=\min_{\ol{\mathcal{N}_L}}\frac{\rho_t}{p'(\rho_t)}>0.\]
Therefore, it follows from \eqref{assumption1-thm1}, Poincar\'{e} inequality, and the Cauchy-Schwarz inequality that if $\|w'\|_{C^0([0,L])}$ is sufficiently small, depending only on $(\mu_0, L)$, then
\[\int_{\mathcal{N}_L}
\frac{1}{\mu_0+L\|w'\|_{C^0([0,L])}}|\nabla(\Phi_1-\Phi_0)|^2
+ \frac{\delta_*}{p'(\varrho^*)}
|\nabla(\varphi_1-\varphi_0)|^2\, d\rx\le 0,\]
from which it immediately follows that 
\[
|\nabla(\Phi_1-\Phi_0)|=|\nabla(\varphi_1-\varphi_0)|=0\quad\text{in $\mathcal{N}_L$}.
\]
Combining the above estimate with the boundary conditions
\begin{equation*}
    \begin{split}
    \Phi_j=\bar{\Phi}(0)+h_0 \quad &\text{on } \Gam_0,\\
    \varphi_j=\bar{\varphi}(L) \quad &\text{on } \Gam_L,
    \end{split}
\end{equation*}
prescribed in \eqref{EP-potential-BC}, we conclude that
\[
(\varphi_0, \Phi_0) = (\varphi_1, \Phi_1)
\quad \text{in } \ol{\mathcal{N}_L}.
\]
This completes the proof of Theorem \ref{theorem-1}.
\end{proof}

\section{Proof of Theorem \ref{theorem-2}}
\label{section:3}
In this section, we study the Euler-Poisson equation of self-gravitation type and prove Theorem \ref{theorem-2}.

Assume that
\begin{equation}
\label{assumption1-thm2}
    \mu_1:=\max_{0\le x_1\le L} w(x_1)<0.
\end{equation}
Then one has
\[
 \mu_1-L\|w'\|_{C^0([0,L])} \le  w(x_1)\le \mu_1\quad\text{on $[0,L]$},
\]
or equivalently
\begin{equation}
\label{estimate of w-s}
\frac{1}{\tilde{\mu}_1+L\|w'\|_{C^0([0,L])}}\le -\frac{1}{w(x_1)}\le  \frac{1}{\tilde{\mu}_1}\quad\text{on $[0,L]$}\,\,\text{for $\tilde{\mu}_1:=-\mu_1$.}
\end{equation}

Let $(\bar{\rho}, {\bar u}_1{\bf e}_1, \bar{\Phi})$ be a smooth subsonic solution of \eqref{EP-1d} on $[0,L]$ for a constant $J>0$. Put
\[
k_0:=\frac 12 \bar{u}_1^2(0)+\frac{\gamma \bar{\rho}^{\gam-1}(0)}{\gamma-1}-\bar{\Phi}(0).
\]
It follows from the continuity equation that there exists a function $\psi$ such that $\rho{\bf u}=\nabla^{\perp}\psi$,
where \[\nabla^{\perp}:=(\der_2, -\der_1)\quad\text{for}\quad
\der_j:=\frac{\der}{\der x_j},\,\,j=1,2.\]
Hence the Euler-Poisson equations can be written as 
\begin{equation}
\label{EP-stream}
    \begin{cases}
\nabla\cdot\left(\frac{\nabla\psi}{\rho}\right)=0\\
    \frac 12 \frac{|\nabla\psi|^2}{\rho^2}+\frac{\gamma \rho^{\gamma-1}}{\gamma-1}-\Phi=k_0\\
    \Delta\Phi=w(x_1)\rho-b(x_1).
    \end{cases}
\end{equation}

Define
\begin{equation}
\label{definition-B-s}
\mathfrak{B}(\rho, {\bf q})
:=\frac{|{\bf q}|^2}{2\rho^2}+\frac{\gamma \rho^{\gamma-1}}{\gamma-1}.
\end{equation}
The straightforward computations show that
\begin{equation}
\label{definition-rhos-s}
\partial_{\rho}\mathfrak{B}(\rho, {\bf q})
\begin{cases}
<0, & \text{if } \rho<\rho_s({\bf q}),\\[2pt]
=0, & \text{if } \rho=\rho_s({\bf q}),\\[2pt]
>0, & \text{if } \rho>\rho_s({\bf q}),
\end{cases}
\quad \text{where}\quad
\rho_s({\bf q})
=\left(\dfrac{|{\bf q}|^2}{\gamma}\right)^{\frac{1}{\gamma+1}}.
\end{equation}
Moreover,
\[
\lim_{\rho\to 0+}\mathfrak{B}(\rho, {\bf q})=\infty \quad \text{if } |{\bf q}|\neq 0,
\qquad
\lim_{\rho\to \infty}\mathfrak{B}(\rho, {\bf q})=\infty.
\]
Therefore, if
\[
k_0+\Phi>\mathfrak{B}(\rho_s(\nabla\psi), \nabla\psi),
\]
then there exist two values $\rho_-$ and $\rho_+$ such that
\[
\rho_-<\rho_s(\nabla\psi)<\rho_+
\quad \text{and} \quad
\mathfrak{B}(\rho_{\pm}, \nabla\psi)-\Phi=k_0.
\]
We denote these two values by
\[
\rho_-=\rho_{\rm sup}(\nabla\psi, \Phi),
\qquad
\rho_+=\rho_{\rm sub}(\nabla\psi, \Phi).
\]
The state $(\rho_{\rm sup}(\nabla\psi, \Phi), \psi, \Phi)$ is supersonic, whereas  
the state $(\rho_{\rm sub}(\nabla\psi, \Phi), \psi, \Phi)$ is subsonic.
Since we seek a subsonic solution in (Problem \ref{main problem}), we rewrite \eqref{EP-stream} in the form
\begin{equation}
\label{EP-stream-final}
    \begin{cases}
    k_0+\Phi-\mathfrak{B}(\rho_s(\nabla\psi), \nabla\psi)>0,\\[2pt]
    \nabla\cdot \left(\dfrac{\nabla\psi}{\rho_{\rm sub}(\nabla\psi, \Phi)}\right)=0,\\[6pt]
    \Delta\Phi=w(x_1)\rho_{\rm sub}(\nabla\psi, \Phi)-b(x_1).
    \end{cases}
\end{equation}

We rewrite the boundary conditions \eqref{elec-bc-1}--\eqref{elec-bc-3} as follows: 
\begin{equation}
\label{EP-stream-BC}
\begin{split}
\psi(0, x_2)=\int_0^{x_2} (J+g_0(t))\,dt,\quad \Phi=\bar{\Phi}(0)+h_0\quad&\mbox{on $\Gam_0$},\\
\psi(x_1, x_2)=\int_0^{x_2} (J+g_0(t))\,dt,\quad \nabla\Phi\cdot {\bf n}_w=0\quad&\mbox{on $\Gam_w$},\\
\nabla\psi\cdot {\bf n}_L=0,\quad \nabla\Phi\cdot {\bf n}_L=\bar{E}(L)+v_L\quad & \mbox{on $\Gam_L$}.
\end{split}
\end{equation}

Theorem \ref{theorem-2} is proved by an argument similar to that of Theorem \ref{theorem-1}.
In \S \ref{subsection:3.1}, following the argument in \cite{bae2015two}, we show that the boundary value problem \eqref{EP-stream-final}--\eqref{EP-stream-BC} admits a solution in a small neighborhood of the background solution provided that \eqref{assumption1-thm2} and \eqref{assumption2-thm2}--\eqref{assumption3-thm2} hold. This establishes part~(a) of Theorem~\ref{theorem-2}.
In \S~\ref{subsection:3.2}, we introduce {\it{the $\lambda$-set} $\mathcal{S}_{\lambda}(\gamma)$} for the stream function formulation~\eqref{EP-stream-final} and show that it is a convex set. This serves as a key ingredient in the proof of part~(b) of Theorem~\ref{theorem-2}.

\subsection{The existence of a solution}
\label{subsection:3.1}
Since $(\bar{\rho}, \bar{u}_1{\bf e}_1, \bar{\Phi})$ is a subsonic solution on $[0,L]$, one has
\begin{equation}
\label{subsonicity-s}
\mathcal{K}_2 := \min_{0\le x_1\le L}\left(\gamma\bar{\rho}^{\gamma-1}-\bar u_1^2\right) >0.
\end{equation}

Define
\begin{equation*}
\bar{\psi}(\rx):=Jx_2\quad\text{for $\rx=(x_1, x_2)\in \ol{\mathcal{N}_L}$.}
\end{equation*}

\begin{prop}
\label{proposition-2}
Assume that \eqref{assumption1-thm2} holds.
Let $\alp\in(0,1)$ be given. There exists a constant $\sigma_1>0$ such that if $w$ and the boundary data $(g_0,h_0, v_L)$ satisfy conditions \eqref{assumption2-thm2} and \eqref{assumption3-thm2} in Theorem~\ref{theorem-2}, then the system \eqref{EP-stream-final} in $\mathcal{N}_L$ with boundary conditions stated in \eqref{EP-stream-BC} admits a solution
\[
(\psi, \Phi)\in \bigl[C^{1,\alp}(\ol{\mathcal{N}_L})\cap C^{2,\alp}(\mathcal{N}_L)\bigr]^2
\]
with the following properties:
\begin{itemize}
\item[(a)] There exists a constant $\lambda>0$ such that
\[
k_0+\Phi-\mathfrak{B}(\rho_s(\nabla\psi), \nabla\psi) \ge \lambda
\quad \text{in } \ol{\mathcal{N}_L}.
\]
\item[(b)] There exists a constant $C>0$ such that
\[
\|(\psi, \Phi)-(\bar{\psi}, \bar{\Phi})\|_{2,\alp,\mathcal{N}_L}^{(-1-\alp,\Gam_w)}
\le
C\Bigl(\|g_0\|_{C^2(\ol{\Gam_0})}+\|h_0\|_{C^2(\ol{\Gam_0})}
+\|v_L\|_{C^2(\ol{\Gam_L})}
\Bigr).
\]
\end{itemize}
The constants $\sigma_1$, $\lambda$, and $C$ depend only on $(\bar{\rho}, \bar{u}_1, \bar{\Phi}, L, \mu_1, \alp)$.
\end{prop}
This proposition establishes part (a) of Theorem~\ref{theorem-2}.

\medskip

For $(z, {\bf q})=(z,q_1, q_2)\in \R\times \R^2$ satisfying
\begin{equation}
\label{nondegeneracy-s}
    k_0+z-\mathfrak{B}(\rho_s({\bf q}), {\bf q})>0,
\end{equation}
define
\begin{equation}
\label{definition:A,B-s}
    {\bf A}(z, {\bf q})=(A_1, A_2)(z,{\bf q}):=\frac{\bf q}{\rho_{\rm sub}({\bf q}, z)},\quad B(z, {\bf q}):=-\rho_{\rm sub}({\bf q},z).
\end{equation}
Suppose that $(\psi, \Phi)\in [C^1(\ol{\mathcal{N}_L})\cap C^2(\mathcal{N}_L)]^2$ solves the boundary value problem \eqref{EP-stream-final}--\eqref{EP-stream-BC}, and set
\begin{equation*}
    (\phi, \Psi):=(\psi, \Phi)-(\bar{\psi}, \bar{\Phi}).
\end{equation*}
The straightforward computations yield the following boundary value problem for $(\phi, \Psi)$:
\begin{equation}
\label{equation-s}
 \begin{cases}
  L_1(\phi, \Psi)={\rm div}{\bf F}(x_1, \Psi, \nabla\phi)\\
  L_2(\phi, \Psi)=f(x_1, \Psi, \nabla\phi)
 \end{cases}\quad \mbox{in $\mathcal{N}_L$},
\end{equation}
with boundary conditions
\begin{align*}
    \phi(x_1, x_2)=\int_0^{x_2} g_0(t)\,dt,\quad
    \Psi=h_0
\quad&\mbox{on $\Gam_0$},\\
\phi(x_1, x_2)=
\int_0^{x_2} g_0(t)\,dt,\quad \nabla\Psi\cdot {\bf n}_w=0\quad &\mbox{on $\Gam_w$},\\
\nabla\phi\cdot {\bf n}_L=0,\quad \nabla\Psi\cdot 
 {\bf n}_L=v_L\quad &\mbox{on $\Gam_L$}
\end{align*}
for
\begin{align}
\label{definition:L-s}
L_1(\phi, \Psi)
    &:={\rm div}\left( \der_{q_j}A_i(\bar{\Phi}, \nabla\bar{\psi})\der_j\phi+\Psi\der_z{\bf A}(\bar{\Phi}, \nabla\bar{\psi})\right),\\
    \notag
L_2(\psi, \Psi)&:=\frac{-1}{w(x_1)}\Delta \Psi-\der_zB(\bar{\Phi}, \nabla\bar{\psi})\Psi-\der_{\bf q}B(\bar{\Phi}, \nabla\bar{\psi})\cdot \nabla \phi,\\
\label{definition:F-s}
{\bf F}(x_1, z, {\bf q}):&=(F_1, \cdots, F_n)(x_1, z, {\bf q})\quad\text{with}\\
\notag
-F_i(x_1, z, {\bf q}):&=\int_0^1 z\left[\der_zA_i(\bar{\Phi}+\tilde z, \nabla\bar{\psi}+\tilde{\bf q})\right]_{(\tilde z, \tilde{\bf q})=(0, {\bf 0})}^{(tz, {\bf q})}+q_j[\der_{q_j}A_i(\bar{\Phi}, \nabla\bar{\psi}+\tilde{\bf q})]_{\tilde{\bf q}={\bf 0}}^{t{\bf q}}\,dt,\\
\label{definition:f-s}
f(x_1, z, {\bf q}):&=\int_0^1 z\left[\der_z B(\bar{\Phi}+\tilde z, \nabla\bar{\psi}+\tilde{\bf q})\right]_{(\tilde z, \tilde{\bf q})=(0, {\bf 0})}^{(tz, {\bf q})}+q_j[\der_{q_j}B(\bar{\Phi}, \nabla\bar{\psi}+\tilde{\bf q})]_{\tilde{\bf q}={\bf 0}}^{t{\bf q}}\,dt,
\end{align}
where $[G(X)]_{X=a}^b:=G(b)-G(a)$. Here, each $\der_j$ denotes $\der_{x_j}$ for $j=1,2$.
\medskip

Differentiating the pseudo-Bernoulli law
\begin{equation}
\label{pseudo Bernoulli-s}
\frac12 \frac{|{\bf q}|^2}{\rho^2}+\frac{\gamma \rho^{\gamma-1}}{\gamma-1}-z=k_0
\end{equation}
with respect to ${\bf q}$ and $z$ for $(z,{\bf q})=(z,q_1,q_2)\in\R\times\R^2$ satisfying \eqref{nondegeneracy-s}, we obtain
\begin{equation}
\label{derivatives of rho-s}
\partial_{\bf q}\rho_{\rm sub}({\bf q},z)
=
\frac{-{\bf q}}{\gamma \rho_{\rm sub}^{\gamma}({\bf q}, z)-\frac{|{\bf q}|^2}{\rho_{\rm sub}({\bf q}, z)}},\quad \partial_{\bf z}\rho_{\rm sub}({\bf q},z)=\frac{1}{\gamma \rho_{\rm sub}^{\gamma-2}({\bf q}, z)-\frac{|{\bf q}|^2}{\rho_{\rm sub}^3({\bf q}, z)}}.
\end{equation}
Then, by \eqref{definition:A,B-s}, a direct computation yields
\begin{equation}
\label{coefficients-s}
\partial_{q_j}A_i(z,{\bf q})
=
\frac{1}{\rho_{\rm sub}}
\left(
\delta_{ij}
+\frac{q_iq_j}{\gamma\rho_{\rm sub}^{\gamma+1}-|{\bf q}|^2}
\right),
\end{equation}
where $\rho_{\rm sub}=\rho_{\rm sub}({\bf q}, z)$.

In particular, for the background solution $(\bar{\Phi},\bar{\psi})$, this identity reduces to
\[
\bar{\rho}\,\partial_{q_j}A_i(\bar{\Phi},\nabla\bar{\psi})=
\begin{cases}
1, & \text{if } i=j=1,\\[2mm]
\dfrac{\gamma\bar{\rho}^{\gamma-1}}{\gamma\bar{\rho}^{\gamma-1}-\bar u_1^2}, & \text{if } i=j=2,\\[2mm]
0, & \text{if } i\neq j.
\end{cases}
\]

By \eqref{subsonicity-s}, there exist positive constants $\lambda_0$ and $\Lambda_0$, depending only on the background solution, such that
\begin{equation}
\label{ellipticity-s}
\lambda_0 \mathbb{I}_2
\le
\bigl[\partial_{q_j}A_i(\bar{\Phi},\nabla\bar{\psi})\bigr]_{i,j=1}^2
\le
\Lambda_0 \mathbb{I}_2
\quad \text{in } \ol{\mathcal{N}_L}.
\end{equation}
Consequently, the bilinear operator $L_1(\cdot,\cdot)$ defined in \eqref{definition:L-s} is uniformly elliptic in $\mathcal{N}_L$ with respect to its first argument.

Furthermore, it follows directly from \eqref{definition:A,B-s} and \eqref{derivatives of rho-s} that
\begin{equation}
\label{essential obser-s}
    \partial_z{\bf A}(z,{\bf q})+\partial_{\bf q}B(z,{\bf q})=0
\end{equation}
for $(z,{\bf q})=(z,q_1,q_2)\in\R\times\R^2$ satisfying \eqref{nondegeneracy-s}. 

Since the identities \eqref{ellipticity-s} and \eqref{essential obser-s} correspond, respectively, to \eqref{ellipticity-p} and \eqref{essential observation-p} in the Euler–Poisson system with self-consistent electric fields, Proposition~\ref{proposition-2} can be proved by following the argument in \cite{bae2015two} and the proof of Proposition~\ref{proposition-1}.
We omit the details. Alternatively, we refer to \cite{bae2015two} for further details.

\subsection{The uniqueness}
\label{subsection:3.2}

To establish part (b) of Theorem~\ref{theorem-2}, we introduce the $\lambda$-set under the assumption \eqref{assumption1-thm2} and show that this set is convex for any $\lambda>0$, provided that $\gamma \ge 3$.
\begin{prop}[Convexity of the {\it{$\lambda$-set}}  ]
\label{proposition-convexity-s}
For $\gamma>1$ and $\lambda\in\R$, define
\[
\mathcal{S}_{\lambda}(\gamma)
:=\left\{(z,{\bf q})\in \R\times \R^2 \middle \vert\;
k_0+z-\mathfrak{B}(\rho_s({\bf q}), {\bf q}) \ge \lambda
\right\}
\]
where $\mathfrak{B}(\rho,{\bf q})$ and $\rho_s({\bf q})$ are given by \eqref{definition-B-s} and \eqref{definition-rhos-s}, respectively.
Then the set $\mathcal{S}_{\lambda}(\gamma)$ is convex for any $\lambda>0$ and $\gamma\ge 3$.

\end{prop}

\begin{proof}
If $\mathcal{S}_{\lambda}(\gamma)=\emptyset$, there is nothing to prove. 
Assume that $(z_j,{\bf q}_j)\in \mathcal{S}_{\delta}(\gamma)$ for $j=0,1$, and define
\[
(z_t,{\bf q}_t):=(1-t)(z_0,{\bf q}_0)+t(z_1,{\bf q}_1), \qquad t\in[0,1].
\]

The straightforward computations using \eqref{definition-B-s} and \eqref{definition-rhos-s} yield
\begin{equation}
\label{sonic representation-s}
\mathfrak{B}(\rho_s({\bf q}),{\bf q})
=\left(\frac12\left(\frac1\gamma\right)^{-\frac{2}{\gamma+1}}
+\frac{\gamma}{\gamma-1}\left(\frac1\gamma\right)^{\frac{\gamma-1}{\gamma+1}}\right)
|{\bf q}|^{\frac{2(\gamma-1)}{\gamma+1}}
=:C(\gamma)|{\bf q}|^{\frac{2(\gamma-1)}{\gamma+1}}.
\end{equation}
For $t\in[0,1]$, define
\begin{align*}
\mathfrak{g}(t)
&:=k_0+z_t-\mathfrak{B}(\rho_s({\bf q}_t),{\bf q}_t) \\
&=k_0+(1-t)z_0+tz_1
-C(\gamma)\bigl|{\bf q}_0+t({\bf q}_1-{\bf q}_0)\bigr|^{\frac{2(\gamma-1)}{\gamma+1}},
\end{align*}
and 
\[
\mathfrak{f}(t)
:=\bigl|{\bf q}_0+t({\bf q}_1-{\bf q}_0)\bigr|^{\frac{2(\gamma-1)}{\gamma+1}}.
\]
Since $(z_j,{\bf q}_j)\in \mathcal{S}_{\lambda}(\gamma)$ for $j=0,1$, it follows that
\[
\min\{\mathfrak{g}(0),\mathfrak{g}(1)\}\ge \lambda.
\]
If
\begin{equation}
\label{condition-gamma}
\frac{2(\gamma-1)}{\gamma+1}\ge 1\left(\Leftrightarrow \gamma\ge 3\right),
\end{equation}
then the function $\mathfrak{f}$ is convex on $[0,1]$. Consequently, $\mathfrak{g}$ is concave on $[0,1]$, and hence
\[
\mathfrak{g}(t)\ge (1-t)\mathfrak{g}(0)+t\mathfrak{g}(1)\ge \delta
\qquad \forall\, t\in[0,1].
\]
Therefore, $(z_t,{\bf q}_t)\in \mathcal{S}_{\lambda}(\gamma)$ for all $t\in[0,1]$ provided that \eqref{condition-gamma} holds. This completes the proof of the proposition.
\end{proof}

\begin{lemma}[Non-vacuum property]
\label{lemma-nonvacuum}
Suppose that $(z,{\bf q})\in \mathcal{S}_{\lambda}(\gamma)$ 
for some $\gamma\in[3,\infty)$ and $\lambda>0$. 
Then there exists a unique constant $\rho^{(z,{\bf q})}$ such that
\begin{equation*}
    \begin{aligned}
    &\rho^{(z,{\bf q})}>\rho_s({\bf q})\quad 
    \text{and}\quad 
    \mathfrak{B}(\rho^{(z,{\bf q})}, {\bf q})=k_0+z,
    \end{aligned}
\end{equation*}
where $\mathfrak{B}(\rho, {\bf q})$ is defined in \eqref{definition-B-s}.
\end{lemma}

\begin{proof}
\textbf{Case 1: $|{\bf q}|=0$.}
In this case, it follows from \eqref{definition-B-s} that
\[
\mathfrak{B}(\rho, {\bf q})
= \frac{\gamma \rho^{\gamma-1}}{\gamma-1},
\]
and from \eqref{sonic representation-s} that
$\rho_s({\bf q})=0$,
so that
$\mathfrak{B}(\rho_s({\bf q}), {\bf q})=0$.
Since $(z,{\bf q})\in \mathcal{S}_{\lambda}(\gamma)$ and $\lambda>0$, we have
$k_0+z>\lambda>0$.
For any $\gamma>1$, the equation
\[
\frac{\gamma \rho^{\gamma-1}}{\gamma-1}
= k_0+z
\]
admits a unique positive solution $\rho=\rho^{(z,{\bf q})}$.
Clearly,
\[
\rho^{(z,{\bf q})}>\rho_s({\bf q})=0.
\]

\medskip
\textbf{Case 2: $|{\bf q}|>0$.}
In this case, the desired conclusion follows from the argument established at the beginning of Section~\ref{section:3}. 
This completes the proof.
\end{proof}

\begin{proof}[Proof for part~(b) of Theorem \ref{theorem-2}]
We will essentially follow the proof of part~(b) of Theorem~\ref{theorem-1} in \S \ref{subsection:uniqueness 1}. In this proof, we assume $\gamma\ge 3.$
The rest of the proof is divided into three steps.

{\it Step 1.}
For $j=0,1$, let
\[
(\rho^{(j)}, {\bf u}^{(j)}, \Phi^{(j)})\in
\bigl[C^0(\ol{\mathcal{N}_L})\cap C^1(\mathcal{N}_L)\bigr]
\times
\bigl[C^0(\ol{\mathcal{N}_L};\R^2)\cap C^1(\mathcal{N}_L;\R^2)\bigr]
\times
\bigl[C^1(\ol{\mathcal{N}_L})\cap C^2(\mathcal{N}_L)\bigr]
\]
be two subsonic solutions of Problem~\ref{main problem} in $\ol{\mathcal{N}_L}$. And, define
\begin{equation*}
\psi^{(j)}(x_1, x_2):=\int_0^{x_2}\left(\rho^{(j)}
{\bf u}^{(j)}\cdot{\bf e}_1\right)(x_1, t)\,dt\quad\text{in $\ol{\mathcal{N}_L}$}.
\end{equation*}
Then we have
\begin{equation}
\label{bvp for uniqueness-s}
\begin{aligned}
    \begin{cases}
    \nabla\cdot ({\bf A}(\Phi^{(1)}, \nabla\psi^{(1)})-{\bf A}(\Phi^{(0)}, \nabla\psi^{(0)}))=0\\
    \Delta (\Phi^{(1)}-\Phi^{(0)})=-w(x_1)(B(\Phi^{(1)}, \nabla\psi^{(1)})-B(\Phi^{(0)}, \nabla\psi^{(0)}))
    \end{cases}\quad& \text{in $\mathcal{N}_L$},\\
    \psi^{(1)}-\psi^{(0)}=0,\quad \Phi^{(1)}-\Phi^{(0)}=0\quad &\text{on $\Gamma_0$},\\
    \psi^{(1)}-\psi^{(0)}=0,\quad \nabla(\Phi^{(1)}-\Phi^{(0)})\cdot {\bf n}_w=0\quad &\text{on $\Gamma_w$},\\
    \nabla(\psi^{(1)}-\psi^{(0)})\cdot {\bf n}_L=0,\quad \nabla(\Phi^{(1)}-\Phi^{(0)})\cdot {\bf n}_L=0\quad &\text{on $\Gamma_L$}
\end{aligned}
\end{equation}
where ${\bf A}(z,{\bf q})$ and $B(z, {\bf q})$ are given in \eqref{definition:A,B-s}.

Define
\begin{align*}
(\Phi^t, \psi^t) &:= (1-t)(\Phi^{(0)}, \psi^{(0)}) + t(\Phi^{(1)}, \psi^{(1)}), \\
\lambda_t &:= \min_{\overline{\mathcal{N}_L}} 
\left(k_0 + \Phi^t - \mathfrak{B}(\rho_s(\nabla \psi^t), \nabla \psi^t)\right).
\end{align*}

Since each $(\rho^{(j)}, \mathbf{u}^{(j)}, \Phi^{(j)})$ is subsonic in 
$\overline{\mathcal{N}_L}$, it follows that $\lambda_j > 0$ for $j=0,1$. 
Hence one has
\[
\lambda_* := \min\{\lambda_0, \lambda_1\}>0.
\]
By Proposition~\ref{proposition-convexity-s}, we then have
\begin{equation}
\label{pre ellipticity-s}
\lambda_t \ge \lambda_* > 0
\quad \text{for all } t \in [0,1].
\end{equation}

{\it Step 2.} It follows from Lemma~\ref{lemma-nonvacuum} and \eqref{pre ellipticity-s} 
that, for each $t \in [0,1]$, 
$\rho_{\rm sub}(\Phi^t, \nabla \psi^t)$ 
is well defined throughout $\overline{\mathcal{N}_L}$. 
Since $\rho_{\rm sub}(z,{\bf q})$ is continuous in $\mathcal{S}_{\lambda_*}(\gamma)$, 
and since the set 
\[
\mathcal{D}=\{(\Phi^t(\mathbf{x}), \nabla \psi^{t}(\mathbf{x})) : \mathbf{x} \in \overline{\mathcal{N}_L},\,\,t\in[0,1]\}
\]
is compact, it follows that
\begin{equation}
\label{uniform nonvacuum-s}
\varrho_* 
:= \min_{t\in[0,1], \atop \rx\in\overline{\mathcal{N}_L}} 
\rho_{\rm sub}(\Phi^t(\rx), \nabla \psi^t(\rx) ) 
> 0.
\end{equation}
In addition, it follows from $\rho_{\rm sub}(z,{\bf q})-\rho_s({\bf q})>0$ in $S_{\lambda_*}(\gamma)$ and the compactness of the set $\mathcal{D}$, we also obtain
\begin{equation}
\label{uniform ellipticity final-s}
\begin{split}
\kappa_*
&:= \min_{t\in[0,1], \atop \rx\in\overline{\mathcal{N}_L}}
\gamma\left(\rho_{\rm sub}^{\gamma+1}(\Phi^t(\rx), \nabla \psi^t(\rx)) -\rho_s^{\gamma+1}(\nabla\psi^t (\rx))\right)\\
&= \min_{t\in[0,1], \atop \rx\in\overline{\mathcal{N}_L}} \left(
\gamma\rho_{\rm sub}^{\gamma+1}(\Phi^t(\rx)) \nabla \psi^t(\rx)) -|\nabla\psi^t(\rx)|^2\right)
> 0.
\end{split}
\end{equation}

The straightforward computations using \eqref{definition:A,B-s} and \eqref{derivatives of rho-s} yield
\begin{align*}
\der_{\bf q}{\bf A}(z,{\bf q})&= \frac{1}{\rho_{\rm sub}(z,{\bf q})} \begin{bmatrix}
\delta_{ij}+\frac{q_iq_j}{\gamma\rho_{\rm sub}^{\gamma+1}(z,{\bf q})-|{\bf q}|^2}
\end{bmatrix}_{i,j\in\{1,2\}},\\
\der_z\rho_{\rm sub}(z,{\bf q})&=\frac{-\rho_{\rm sub}(z,{\bf q})}{\gamma\rho_{\rm sub}^{\gam-1}(z,{\bf q})-\frac{|{\bf q}|^2}{\rho_{\rm sub}^2(z,{\bf q})}}
\end{align*}
for $(z,{\bf q})\in \mathcal{S}_{\lambda_*}(\gamma)$.
It follows from \eqref{uniform nonvacuum-s} and 
\eqref{uniform ellipticity final-s} that there exists a constant 
$\mu_* > 0$, depending only on $\varrho_*$ and $\kappa_*$, such that
\begin{equation}
\label{essentials in estimate-s}
\mu_* \mathbb{I}_2 
\le \partial_{\mathbf q}\mathbf A(\Phi^t, \nabla \psi^t) 
\le \frac{1}{\mu_*}\mathbb{I}_2,
\qquad
\partial_z \rho_{\rm sub}(\Phi^t, \nabla \psi^t) 
\le -\mu_*
\quad \text{in } \overline{\mathcal{N}_L},
\end{equation}
for all $t \in [0,1]$.
\medskip

{\it Step 3.} Put
\begin{equation*}
\phi:=\psi_1-\psi_0,\quad \Psi:=\Phi_1-\Phi_0.
\end{equation*}
And we rewrite \eqref{bvp for uniqueness-s} as
\begin{equation*}
\begin{aligned}
 \begin{cases}
E_1:=    \nabla\cdot \left(\int_0^1\der_i\phi \der_{q_i}{\bf A}_j(\Phi^t, \nabla\psi^t)+ \Psi \der_z {\bf A}_j(\Phi^t, \nabla\psi^t)\,dt\right)=0\\
E_2:=    -\frac{1}{w(x_1)}\Delta \Psi- (\Psi, \nabla\phi)\cdot \int_0^1(\der_z,\der_{\bf q})B(\Phi^t, \nabla\psi^t)\,dt=0\end{cases}\quad &\text{in $\mathcal{N}_L$,}\\
   \phi=0,\quad \Psi=0\quad &\text{on $\Gamma_0$},\\
    \phi=0,\quad \nabla\Psi\cdot {\bf n}_w=0\quad &\text{on $\Gamma_w$},\\
    \nabla\phi\cdot {\bf n}_L=0,\quad \nabla\Psi\cdot {\bf n}_L=0\quad &\text{on $\Gamma_L$}.   
    \end{aligned}
\end{equation*}
Integrating by parts yields
\begin{equation*}
\begin{split}
0&=\int_{\mathcal{N}_L} E_1\phi+E_2\Psi\,d\rx\\
&=\int_{\Gamma_L} \phi \int_0^1\der_i\phi\der_{q_i}{\bf A}_1(\Phi^t, \nabla\psi^t)+\Psi\der_z{\bf A}_1(\Phi^t, \nabla\psi^t)\,dt\,dx_2\\
&\phantom{=}-\int_{\mathcal{N}_L} \der_i\phi\der_j\phi \int_0^1 \der_{q_i}{\bf A}_j(\Phi^t, \nabla\psi^t)\,dt\,d\rx-\int_{\mathcal{N}_L} \Psi\der_j\phi \int_0^1\der_z{\bf A}_j(\Phi^t, \nabla\psi^t)\,dt\,d\rx\\
&\phantom{=}-\int_{\mathcal{N}_L} \nabla\Psi\cdot \nabla\left(\frac{-1}{w(x_1)} \Psi\right)\,d\rx\\
&\phantom{=}-\int_{\mathcal{N}_L} \Psi^2\int_0^1 \der_z B(\Phi^t, \nabla\psi^t)\,dt\,d\rx-\int_{\mathcal{N}_L} \Psi \nabla\phi\cdot \int_0^1\der_{\bf q}B(\Phi^t, \nabla\psi^t)\,dt\,d\rx
\end{split}
\end{equation*}

From
\begin{equation*}
\begin{aligned}
\der_1\phi=\nabla\phi\cdot {\bf n}_L=0\quad&\mbox{on $\Gamma_L$,}\\
\der_2{\bf A}_2(\Phi^t, \nabla\psi^t)
=\frac{\der_1\psi^t\der_2\psi^t}
{\rho_{\rm sub}({\bf P}^t)\left(\gamma\rho_{\rm sub}^{\gamma+1}({\bf P}^t)-|\nabla\psi^t|^2\right)}
=0\quad&\mbox{on $\Gamma_L$,}
\end{aligned}
\end{equation*}
for ${\bf P}^t:=(\Phi^t, \nabla\psi^t)$, it follows that
\[
\int_{\Gamma_L} \phi \int_0^1 \der_i\phi\,\der_i{\bf A}_1(\Phi^t, \nabla\psi^t)\,dt\,dx_2=0.
\]
Furthermore, by \eqref{derivatives of rho-s} and \eqref{essential obser-s}, we obtain
\[
\der_z{\bf A}_1(\Phi^t, \nabla\psi^t)
={\bf e}_1\cdot\der_{\bf q}\rho_{\rm sub}({\bf P}^t)
=\frac{\der_1\psi^t}
{\gamma\rho_{\rm sub}^{\gamma}({\bf P}^t)-\frac{|\nabla\psi^t|^2}{\rho_{\rm sub}({\bf P}^t)}}
=0
\quad\text{on $\Gamma_L$.}
\]
Therefore,
\begin{equation}
\label{integral 1-s}
\int_{\Gamma_L} \phi \int_0^1 \Psi\,\der_z{\bf A}_1(\Phi^t, \nabla\psi^t)\,dt\,dx_2=0.
\end{equation}
Finally, using \eqref{derivatives of rho-s} and \eqref{essential obser-s} again yields
\begin{equation}
\label{integral 2-s}
\int_{\mathcal{N}_L} \Psi\,\der_j\phi
\int_0^1 \der_z{\bf A}_j(\Phi^t, \nabla\psi^t)\,dt\,d\rx
+
\int_{\mathcal{N}_L} \Psi\, \nabla\phi\cdot
\int_0^1 \der_{\bf q}B(\Phi^t, \nabla\psi^t)\,dt\,d\rx
=0.
\end{equation}

It follows from \eqref{essentials in estimate-s} that one has 
\begin{equation}
\label{integral 3-s}
\begin{aligned}
&\int_{\mathcal{N}_L} \der_i\phi\der_j\phi \int_0^1 \der_{q_i}{\bf A}_j(\Phi^t, \nabla\psi^t)\,dt\,d\rx+\int_{\mathcal{N}_L} \Psi^2\int_0^1 \der_z B(\Phi^t, \nabla\psi^t)\,dt\,d\rx\\
\ge & \mu_*\int_{\mathcal{N}_L}|\nabla\phi|^2+\Psi^2\,d\rx.
\end{aligned}
\end{equation}

Integration by parts together with \eqref{estimate of w-s} gives
\begin{equation}
    \label{integral 4-s}
    \int_{\mathcal{N}_L}\nabla \Psi\cdot \nabla\left(-\frac{1}{w(x_1)}\Psi\right)\,d\rx
    \ge 
    \int_{\mathcal{N}_L}\frac{1}{\tilde{\mu}_1}|\nabla\Psi|^2
    -\frac{\|w'\|_{C^0[0,L]}}{\tilde{\mu}_1+L\|w'\|_{C^0[0,L]}}|\Psi\der_1\Psi|\,d\rx .
\end{equation}
Combining \eqref{integral 1-s}--\eqref{integral 4-s}, we obtain
\begin{equation*}
\begin{aligned}
0&=\int_{\mathcal{N}_L} E_1\phi+E_2\Psi\,d\rx\\
&\le -\int_{\mathcal{N}_L}\mu_*
\left(|\nabla\phi|^2+\Psi^2\right)+\frac{1}{\tilde{\mu}_1}|\nabla\Psi|^2\,d\rx
-\frac{\|w'\|_{C^0[0,L]}}{\tilde{\mu}_1+L\|w'\|_{C^0[0,L]}}
\int_{\mathcal{N}_L} |\Psi\der_1\Psi|\,d\rx .
\end{aligned}
\end{equation*}
With the aid of the Cauchy--Schwarz inequality, if $\|w'\|_{C^0[0,L]}$ is sufficiently small depending only on $(\mu_*, \tilde{\mu}_1, L)$, then we have
\[
\int_{\mathcal{N}_L}\mu_*
\left(|\nabla\phi|^2+\Psi^2\right)
+\frac{1}{\tilde{\mu}_1}|\nabla\Psi|^2\,d\rx=0.
\]
Hence,
$\phi=\Psi=0$ in $\ol{\mathcal{N}_L}$.
This completes the proof of Theorem \ref{theorem-2}.
\end{proof}

\medskip
\noindent
{\bf Acknowledgments:}
The research of Myoungjean Bae was supported in part by the Ministry of Education of the Republic of Korea and the National Research Foundation of Korea (NRF-RS-2025-00553734). The research of Ben Duan was supported in part by National Key R\&D Program of China No. 2024YFA1013303, NSFC Grant No. 12271205. And, the research of Chunjing Xie was partially supported by NSFC grants 12571238 and 12426203. The third author thanks helpful discussions with Professor David Ruiz.

\bigskip
\bibliographystyle{plain}
\bibliography{References_EP_smooth_tr_new}

\end{document}